\documentclass[11pt]{article}

\usepackage{amssymb,graphics,amsmath,amsthm,amsopn,amstext,amsfonts,color,fullpage,fancybox,hyperref,bm}
\usepackage{subfig,placeins} 
\usepackage{graphicx,color,float,epstopdf, authblk}
\usepackage{xcolor}

\usepackage[toc,page]{appendix}
\newtheorem{theorem}{Theorem}
\newtheorem{proposition}[theorem]{Proposition}
\newtheorem{lemma}[theorem]{Lemma}
\newtheorem{corollary}[theorem]{Corollary}
\newtheorem{definition}[theorem]{Definition}
\newtheorem{example}[theorem]{Example}
\newtheorem{remark}[theorem]{Remark}

\usepackage{amssymb,amsmath, mathtools, algpseudocode, epsfig}
\usepackage{graphicx, psfrag, url, bm, paralist, color, algorithm, subfig}

\algnewcommand{\IfThenElse}[3]{
	\State \algorithmicif\ #1\ \algorithmicthen\ #2\ \algorithmicelse\ #3}

\newcommand{\rank}{{\rm rank}}

\newcommand{\st}{{\rm s.t.}}

\newcommand{\ball}{{\rm I\!B}}

\newcommand{\bbW}{\bar{\mathbf{W}}}
\newcommand{\bbw}{\bar{\mathbf{w}}}
\newcommand{\bbR}{\bar{\mathbf{R}}}
\newcommand{\bbA}{\bar{\mathbf{A}}}
\newcommand{\tbW}{\widetilde{\mathbf{W}}}
\newcommand{\tbw}{\widetilde{\mathbf{w}}}
\newcommand{\bzero}{\mathbf{0}}

\newcommand{\bsigma}{\boldsymbol{\sigma}}

\newcommand{\bSigma}{\boldsymbol{\Sigma}}

\newcommand{\bDelta}{\boldsymbol{\Delta}}

\newcommand{\ba}{\mathbf{a}}
\newcommand{\bb}{\mathbf{b}}
\newcommand{\bc}{\mathbf{c}}

\newcommand{\bp}{\mathbf{p}}

\newcommand{\bv}{\mathbf{v}}
\newcommand{\bw}{\mathbf{w}}

\newcommand{\bz}{\mathbf{z}}
\newcommand{\bA}{\mathbf{A}}
\newcommand{\bB}{\mathbf{B}}

\newcommand{\bI}{\mathbf{I}}

\newcommand{\bM}{\mathbf{M}}
\newcommand{\bN}{\mathbf{N}}

\newcommand{\bP}{\mathbf{P}}
\newcommand{\bQ}{\mathbf{Q}}
\newcommand{\bR}{\mathbf{R}}
\newcommand{\bS}{\mathbf{S}}
\newcommand{\bT}{\mathbf{T}}
\newcommand{\bU}{\mathbf{U}}
\newcommand{\bV}{\mathbf{V}}
\newcommand{\bW}{\mathbf{W}}
\newcommand{\bX}{\mathbf{X}}
\newcommand{\bY}{\mathbf{Y}}
\newcommand{\bZ}{\mathbf{Z}}

\providecommand{\keywords}[1]
{
  \small	
  \textbf{\textit{Keywords---}} #1
}

\newtheorem{observation}{Observation}
\newtheorem{property}{Property}

\title{Learning Deep Models:\\ Critical Points and Local Openness}

\author[1]{Maher Nouiehed\thanks{mn102@aub.edu.lb}}
\author[2]{Meisam Razaviyayn\thanks{razaviya@usc.edu}}
\affil[1]{Department of Industrial Engineering and Management,  American University of Beirut}
\affil[2]{Department of Industrial and Systems Engineering,University of Southern California}

\date{March 1, 2023}

\begin{document}

\maketitle

\begin{abstract}
        With the increasing popularity of non-convex \textit{deep models}, developing a unifying theory for studying the optimization problems that arise from training these models becomes very significant. Toward this end, we present in this paper a unifying landscape analysis framework that can be used when the training objective function is the composite of simple functions.\\

	Using the \textit{local openness} property of the underlying training models,  we provide simple sufficient conditions under which any local optimum of the resulting optimization problem is  globally optimal. We first \textit{completely characterize the local openness of the symmetric and non-symmetric matrix multiplication mapping 
	}. Then we use our characterization to: 1) provide a simple proof for  the classical result of Burer-Monteiro and extend it to non-continuous loss functions. 2) Show that every local optimum of two layer linear networks is globally optimal. Unlike many existing results in the literature, our result requires no assumption  on the target data matrix $\bY$, and input data matrix $\bX$. 3) Develop a \textit{complete} characterization of the local/global optima equivalence of multi-layer linear neural networks. We provide various counterexamples to show the necessity of each of our assumptions. 4) Show global/local optima equivalence of over-parameterized non-linear deep models having a certain pyramidal structure. In contrast to existing works, our result requires no assumption on the differentiability of the activation functions and can go beyond ``full-rank" cases.
 \end{abstract}


\newpage
\section{Introduction}
Deep learning is an inference tool that has recently led to significant practical success in various fields ranging from computer vision to natural language processing. Despite its wide empirical use, the theoretical understanding of the landscapes of the optimization problems corresponding to the underlying neural network architecture models is still very limited. While some recent works have tried to explain these successes through the lens of \textit{expressivity} by showing the power of these models in learning large class of mappings, other works find the root of the success in the \textit{generalizability} of these models from learning perspective.\\

From an optimization perspective, training deep models requires solving non-convex optimization problems, where non-convexity arises from the ``deep'' structure of the model. In fact, it has been shown by  \cite{blum1989training} that training neural networks to global optimality is NP-complete in the worst case (even for a simple three-node network). Despite this worst case barrier, the practical success of deep learning may {\color{black} suggest a special structure in the landscapes of the underlying optimization problems.}   In particular, \cite{choromanska2015loss} uses spin glass theory and empirical experiments to show that local optima of deep neural network optimization problems are close to the global optima.\\


In an effort to better understand the landscape of training deep neural networks, \cite{kawaguchi2016deep, lu2017depth, yun2017global, hardt2016identity, laurent2018deep, zhang2019depth, freeman2016topology} studied deep linear neural networks and provided sufficient conditions under which critical points (or local optimal points) of the training optimization problems are globally optimal. {\color{black} Particularly, \cite{freeman2016topology} and \cite{laurent2018deep} show that every local optimum of a deep linear neural network is globally optimal when the widths of intermediate layers are wider than those of the input or output layer.} The local/global equivalence for deep linear neural networks was also established by \cite{lu2017depth} under the assumption that the input matrix $\bX$ and label matrix $\bY$ are both full row rank. Under similar assumptions, \cite{yun2017global} show that every critical point of a deep linear network is a global optimum. \\

By providing multiple examples of non-linear network structures, \cite{yun2018small, ding2019sub} show that this local/global equivalence cannot be generally extended to deep non-linear networks. Despite the existence of spurious local minima in non-linear networks, multiple works have shown that with over-parameterization and proper random initialization, local optima of the resulting optimization problems can be computed using local search procedures. These results are algorithm-dependent and require certain assumptions on the distribution of the input data. Moreover, such results either assume specific activation functions \cite{venturi2018spurious, soltanolkotabi2019theoretical, du2018power, nguyen2019connected} or apply to semi-unrealistic wide networks 
\cite{allen2018convergence, li2018learning, allen2018learning, du2018gradient, zou2018stochastic, arora2018convergence,li2018over, nguyen2018loss, arora2019exact, ding2019sub}. {\color{black} For instance, \cite{li2018over, nguyen2018loss} showed that no set-wise strict local minima exist when the last layer has more neurons than the number of samples.} With modest over-parameterization, \cite{oymak2019towards} show that (stochastic) gradient descent with random initialization converges to a nearly global solution for neural networks with smooth activation functions.\\

Despite the abundance in research studying the landscape of deep optimization problems, many of the existing results are problem-specific or algorithmic-dependent. In this paper, we develop a \textit{unifying theoretical framework for studying the landscape of non-convex optimization problems that are composition of multiple simple mappings}. The theoretical framework harnesses the concept of \textit{local openness} from differential geometry to provide \textit{sufficient conditions under which local optima of the objective function are globally optimal}. Specifically, consider the general optimization problem
\begin{equation}\label{Prob_general}
\min_{\bw \in \mathcal{W}} \; \ell(\mathcal{F}(\bw)),
\end{equation}
where $\mathcal{F}: \mathcal{W} \mapsto \mathcal{Z}$ is a mapping and $\ell: \mathcal{Z}\mapsto \mathbb{R}$ is the loss function. We define the auxiliary optimization problem
\begin{equation}\label{Prob_general_formulated}
\min_{\bz \in \mathcal{Z}} \; \ell(\bz),
\end{equation}
where $\mathcal{Z}$ is the range of the mapping $\mathcal{F}$. In this paper we analyze the local openness of several popular mapping ${\cal F}$ to establish a connection between the local optima of the optimization problems (\ref{Prob_general}) and (\ref{Prob_general_formulated}). This connection is then used to study the local/global equivalence of (\ref{Prob_general}). Our contributions are summarized below:

{\color{black}\begin{itemize}
    \item We completely characterize the local openness of the matrix multiplication mappings ${\cal M}(\bW_1, \bW_2) \triangleq \bW_1 \bW_2$ and ${\cal M}_{+}(\bW) \triangleq \bW \bW^{\top}$; extending the results shown in~\cite{behrends2017matrix}. These mappings naturally appear in several practical problems such as non-convex matrix factorization, Burer-Monteiro approach for semi-definite programming, training deep neural networks, and matrix completion. Our openness results were utilized to establish the local/global equivalence in several non-convex models including low-rank matrix recovery, matrix completion, multi-linear neural networks, and hierarchical non-linear deep neural networks.
    
    \item  We provide a set of necessary and sufficient conditions on the architecture of multi-linear neural networks for local/global optima equivalence. When those conditions are met, all local optimal points are globally optimal; otherwise there exists local optima that are not global. Unlike many existing results, our framework requires no assumptions on the adopted optimization algorithm nor on the probability distribution of the input data. {\color{black} More specifically, our method analyzes the underlying landscapes by studying the structure of the network regardless of what algorithm or input data are adopted for training.} Moreover, all conditions required for local/global equivalence in existing literature are satisfied by our conditions.
    
    \item We use our framework to study the local/global equivalence of non-linear neural networks with pyramidal structure; networks with deeper layers having a lower number of neurons. We show that for continuous and strictly monotone activation functions, every local minimum $\bar{\bW}$ with all weight matrices $\bar{\bW}_i$'s being full row rank is a global minimum. Unlike the results in~\cite{nguyen2017loss}, our results hold for non-differentiabile activation and loss functions.
\end{itemize}}

{\color{black} The rest of the paper is organized as follows. In Section \ref{Sec:Framework}, we detail our proposed framework that utilizes local openness property to provide simple sufficient conditions under which deep learning optimization problems satisfy local/global optima equivalence.  In Section \ref{Sec:Matrix_Mul}, we completely characterize the local openness of matrix multiplication mappings that naturally appear in deep learning models. In Section~\ref{sec:non-linear-local-global}, we show local/global optima equivalence for non-linear deep models having a certain pyramidal structure. Moreover, in Section~\ref{sec:two-layer-results}, we study the equivalence of local and global optima in two layer linear networks. We extend our study to multi-layer linear neural networks in Section~\ref{sec:multi-layer-results}. Finally, Section \ref{Sec:Con} concludes the paper with a brief discussion.} Before proceeding with our results, we define the following notation.

\subsection{Notation}
First, we use ${\bA}_{l,:}$ and ${\bA}_{:,l}$ to denote the $l^{th}$ row and $l^{th}$ column of the matrix $\bA$ respectively. We denote by {\color{black}$\bI$ the identity matrix} and by $\bI_d \in \mathbb{R}^{d \times d}$ the $d \times d$-dimensional identity matrix. Let $\|\bA\|$, $\mathcal{N}(\bA)$, $\mathcal{C}(\bA)$, $\rank(\bA)$ be respectively the Frobenius norm, null-space, column-space, and rank of the matrix $\bA$. Given subspaces $\bU$ and $\bV$, we say $\bU \perp \bV$ if $\bU$ is orthogonal to $\bV$, and $\bU={\bV}^{\perp}$ if $\bU$ is the orthogonal complement of $\bV$. We say matrix $\bA \in \mathbb{R}^{d_1 \times d_0}$ is rank \textit{deficient} if rank($\bA)< \min \{d_1,d_0\}$, and \textit{full rank} if $\rank(\bA)=\min\{d_1,d_0\}$. We call a point $\bW=({\bW}_h, \ldots , {\bW}_1)$, with $\bW_i \in \mathbb{R}^{d_i \times d_{i-1}}$, \textit{non-degenerate} if rank$(\prod_{i=1}^h \bW_i)=\min_{0\leq i \leq h}\, \, d_i$, and \textit{degenerate} if $\rank(\prod_{i=1}^h \bW_i) <  \min_{0\leq i \leq h}\, \, d_i$. We also say a point $\bW$ is a second order saddle point (strict saddle point) of an unconstrained optimization problem if the gradient of the objective function is zero at $\bW$ and the Hessian of the objective function at $\bW$ has a negative eigenvalue.

\section{Proposed Framework}\label{Sec:Framework}
Consider the general optimization problem defined in~\eqref{Prob_general} and its corresponding auxiliary problem~\eqref{Prob_general_formulated}. Since problem (\ref{Prob_general_formulated}) minimizes the function $\ell(\cdot)$ over the range of the mapping $\mathcal{F}$, the global optimal objective values of problems (\ref{Prob_general}) and (\ref{Prob_general_formulated}) are the same. Moreover, there is a clear relation between the global optimal points of the two optimization problems through the mapping $\mathcal{F}$. However, the connection between the local optima of the two optimization problems is not clear. This connection, in particular, is important when the local optima of \eqref{Prob_general_formulated} are {\color{black} either globally optimal or close to optimal.} In what follows, we establish the connection between the local optima of the optimization problems (\ref{Prob_general}) and (\ref{Prob_general_formulated}) under simple sufficient conditions.  This connection is then used to study the relation between local and global optima of (\ref{Prob_general}) and (\ref{Prob_general_formulated}) for various non-convex learning models. In summary, our strategy is to map the original optimization problem~\eqref{Prob_general} to a generally simpler auxiliary problem~\eqref{Prob_general_formulated} for which the underlying landscape has a special structure (example all local optima are global). {\color{black} We then show that this special structure holds for~\eqref{Prob_general}} by establishing the connection between the local optima of the two problems. We start by defining the following important concepts:
\begin{definition}
{\color{black} \textbf{Relative openness:} Consider a set ${\color{black}{\cal S}} \in \mathbb{R}^n$. A subset ${\cal U} \subseteq {\cal S}$ is said to be an open set relative to ${\cal S}$ if for any $u \in {\cal U}$, there exits $\delta >0$ such that $\ball_{\delta}(u)\, \cap \, {\cal S} \subseteq {\cal U}$.}
\end{definition}

\begin{definition}
\textbf{Open mapping:} A mapping $\mathcal{F}: \mathcal{W} \rightarrow \mathcal{Z}$ is said to be open, if for every open set $U \in \mathcal{W}$, $\mathcal{F}(U)$ is relatively open in $\mathcal{Z}$.
\end{definition}

\begin{definition}
\textbf{Locally open mapping:} A mapping $\mathcal{F}(\cdot)$ is said to be locally open at $\bw$ if for every $\epsilon >0$, there exists $\delta>0$ such that {\color{black}$\ball_{\delta}\big(\mathcal{F}(\bw)\big)  \, \cap \, {\cal R}_{\cal F} \subseteq \, \mathcal{F}\big(\ball_{\epsilon}(\bw)\big)$}. 
\end{definition}
Here {\color{black} $\ball_{\epsilon}(\bw)\subseteq \mathcal{W}$} is an open ball with radius {\color{black}${\epsilon}$} centered at $\bw$, {\color{black}$\ball_{\delta}(\mathcal{F}(\bw))\subseteq\mathcal{Z}$} is a  ball of radius {\color{black}$\delta$} centered at $\mathcal{F}(\bw)$, {\color{black}and ${\cal R}_{\cal F}$ is the range (imageset) of the mapping ${\cal F}$}. Intuitively, we say a mapping is locally open at $\bw$ if for any small perturbation $\widetilde{\bz} \in {\cal R}_{\cal F}$ of $\bar{\bz}\, = {\cal F}(\bw)$, there exists $\widetilde{\bw}$ {\color{black} which is} a small perturbation of $\bw$, such that $\widetilde{\bz}\,= {\cal F}(\widetilde{\bz}).$ {\color{black} This directly implies that all invertible functions are locally open at every point.} By definition, openness of a mapping is stronger than local openness. Furthermore, it {\color{black} directly follows} that a mapping is locally open everywhere if and only if it is open.  A useful property of (locally) open mappings is stated below.

{\color{black}\begin{property}\label{prop:composition}
The composition of two (locally) open maps is (locally) open at a given point.
\end{property}}
\vspace{0.2cm}

The following simple intuitive observation which establishes the connection between the local optima of (\ref{Prob_general}) and (\ref{Prob_general_formulated}),  is a major building block of our analyses.

{\color{black}
\begin{observation}\label{lm:LocalMinMapping}
	Suppose that $\mathcal{F}(\cdot)$ is locally open at $\bbw$. If $\bbw$ is a local minimum of problem (\ref{Prob_general}), then $\bar{\bz}=\mathcal{F}(\bbw)$ is a local minimum of problem (\ref{Prob_general_formulated}). Furthermore, if all local minima of the auxiliary problem~(\ref{Prob_general_formulated}) are global, then every local minimum of~\eqref{Prob_general} is globally optimal. 
\end{observation}
\begin{proof}
Let $\bbw$ be a local minimum of problem (\ref{Prob_general}). Then there exists an $\epsilon>0$ such that $\ell(\mathcal{F}(\bbw)) \leq \ell(\mathcal{F}(\bw)), \, \, \forall \,\bw \in \ball_{\epsilon}(\bbw)$. By the definition of local openness,
	$	\exists \, \delta>0$ such that $\ball_{\delta}(\bar{\bz}) \, \cap {\cal R}_{\cal F} \subseteq \, \mathcal{F}\big(\ball_{\epsilon}(\bbw)\big)$ with $\bar{\bz}={\cal F}({\bbw})$. Therefore, $ \ell(\bar{\bz}) \leq \ell(\bz), \, \, \forall \, \bz \in \ball_{\delta}(\bar{\bz})\,\cap \, {\cal R}_{\cal F}$,
	which implies $\bar{\bz}$ is a local minimum of problem (\ref{Prob_general_formulated}). Furthermore, assume $\bbw$ is a local minimum of~\eqref{Prob_general} and let $\ell_{min}$ be the optimal objective value of problems~\eqref{Prob_general}~and~\eqref{Prob_general_formulated}, then 
	\[\ell \left({\cal F}(\bbw)\right) = \ell(\bar{\bz}) =\ell_{min},\]
	where the second equality holds by assuming that every local minimum of~\eqref{Prob_general_formulated} is global. We conclude that $\bbw$ is a global minimum of~\eqref{Prob_general}.

\end{proof}

The above  observation can be used to map multiple local optima of the original problem~\eqref{Prob_general} to one local optimum of the auxiliary problem~\eqref{Prob_general_formulated}; and potentially make the problem easier to analyze. This mapping is particularly interesting in neural networks since permuting the neurons and the corresponding weights in each layer does not change the objective function. Hence, by nature, the underlying landscapes of these optimization problems have multiple local/global optima. However, collapsing these multiple local optima to one could potentially simplify the problem. In other words, instead of analyzing the original landscape with multiple disconnected local optima, we analyze the landscape of the auxiliary problem. Let us clarify this point through the following simple examples: 
\begin{figure}[H]
	\centering
	\subfloat[Original Problem]{\label{fig:a}\includegraphics[scale = 0.55]{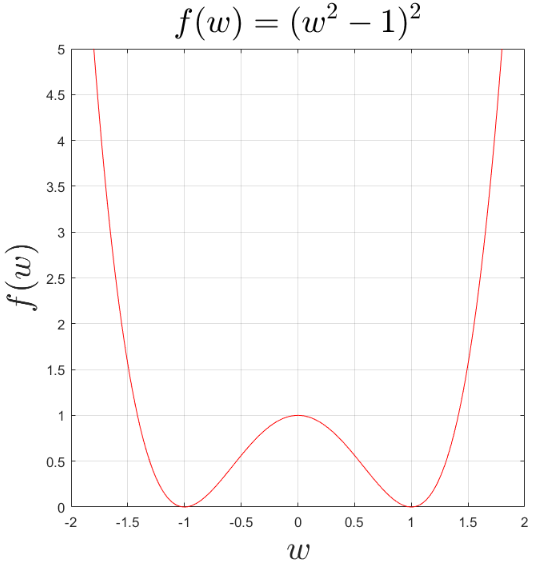}}
	\hspace{1 cm}
	\subfloat[Auxiliary Problem]{\label{fig:b}\includegraphics[scale = 0.55]{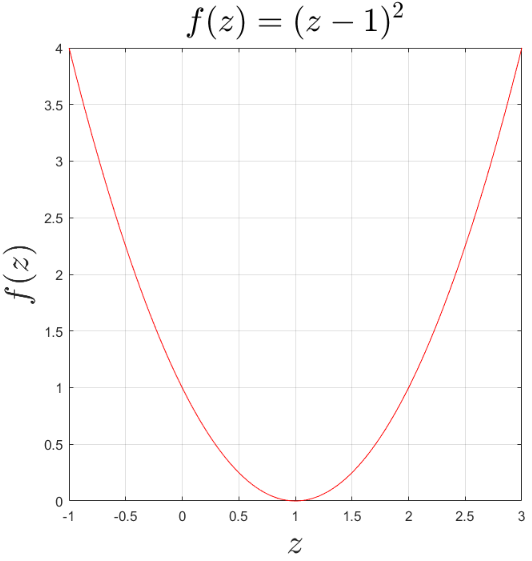}}
	\caption{Two local minima $w =-1$ and $w = +1$ in (a) are mapped to a single local minimum $z = 1$ in (b).}
\end{figure}
\begin{example}
	Consider the optimization problem 
	\begin{equation}\label{LocalMinMapExample}
	\min_{w \in \mathbb{R}} \quad (w^2 - 1)^2,
	\end{equation}
	and its corresponding auxiliary problem
	\begin{equation}\label{LocalMinMapExampleAux}
	\displaystyle{\min_{z \geq 0}} \,\, (z - 1)^2.
	\end{equation}	
	The plots of these two  problems can be found in Figure~\ref{fig:a} and Figure~\ref{fig:b}. Since $\mathcal{F}(w) \triangleq w^2$ is an open mapping in its range, it follows from Observation~\ref{lm:LocalMinMapping} that every local minimum  in problem (\ref{LocalMinMapExample}) is {\color{black} mapped} to a local minimum of problem (\ref{LocalMinMapExampleAux}). Thus the two local minima $w= -1$ and $w= +1$   in (\ref{LocalMinMapExample}) are mapped to a single local minimum $z=1$  of problem (\ref{LocalMinMapExampleAux}). Moreover, since the  optimization problem \eqref{LocalMinMapExampleAux} is convex, the local minimum is global; and hence the original local optima $w=-1$ and $w = +1$ should be both global despite non-convexity of~\eqref{LocalMinMapExample}.
\end{example}

\begin{example}
	Another example is related to the widely used matrix multiplication mapping $\bW_1\bW_2$. Let $(\bbW_1,\bbW_2)$ be a local minimum of the optimization problem 
	\[\min_{\bW_1, \bW_2 } \, \ell(\bW_1\bW_2).\] 
	Then, any point in the set $ {\cal S} \triangleq \{(\bbW_1\bQ_1,\bQ_2\bbW_2)$ with $\bQ_1\bQ_2=\bI\}$ is also a local minimum. If the matrix product $ \bW_1\bW_2$ is locally open at the point $(\bbW_1,\bbW_2)$, then all points in ${\cal S}$ are mapped to a single local minimum $\bar{\bZ}= \bbW_1\bbW_2$ in the corresponding auxiliary problem. A simple {\color{black} one-dimensional} example is plotted in Figures~\ref{fig:c} and \ref{fig:d}.
	
	\begin{figure}[ht] \label{example2-multiplie local to one local}  
    \centering  
	\subfloat[Original Problem]{\label{fig:c}\includegraphics[scale = 0.55]{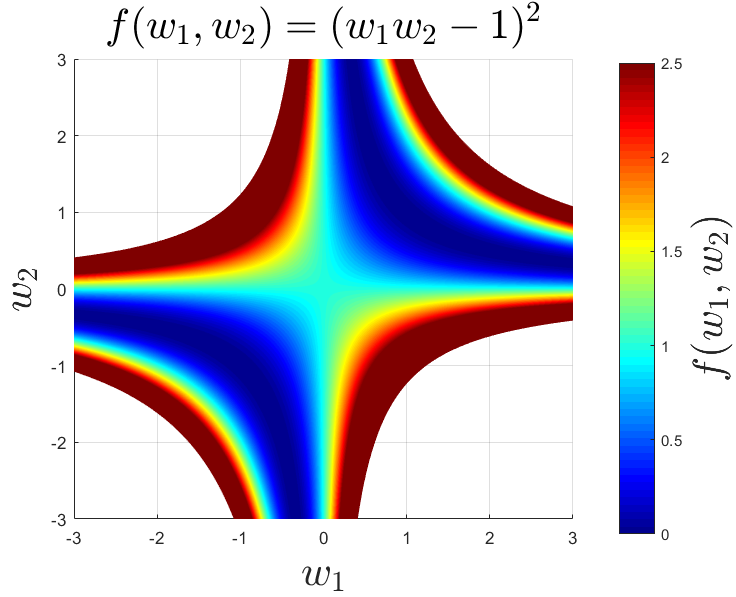}}
	\hspace{0.5 cm}
	\subfloat[Auxiliary Problem]{\label{fig:d}\includegraphics[scale = 0.55]{Plot_4_new.PNG}}
	\caption{All the points in the set $ \{(w_1,w_2) \, | \, w_1w_2 = 1\}$  are local minima in (a) and are mapped to a single local minimum $z = 1$ in (b).}
\end{figure}
\end{example}

\begin{example}
    To demonstrate the sufficiency of the local openness property in relating the local optima of the original and constructed auxiliary problem, we provide the following simple one-dimensional example. Consider the optimization problem
    \begin{equation}\label{LocalMinMapExample2}
	\min_{w \in \mathbb{R}} \quad \left(\left(w^3  - 2w^2 + w -1 \right) - 1 \right)^2,
	\end{equation}
    and its corresponding auxiliary problem
	\begin{equation}\label{LocalMinMapExampleAux2}
	\displaystyle{\min_{z \in \mathbb{R}}} \,\, (z - 1)^2.
	\end{equation}	
	Plots of these two problems can be found in Figure~\ref{fig:e} and Figure~\ref{fig:f}. Since ${\cal F}(w) \triangleq w^3  - 2w^2 + w -1 $ is not locally open at the local optimum $w = 1/3$, the point ${\cal F}(1/3) = -23/27$ is not a local optimum of the auxiliary problem. 
	\begin{figure}[ht] \label{example3-local-optima-to-non-local-optima}  
    \centering  
	\subfloat[Original Problem]{\label{fig:e}\includegraphics[scale = 0.465]{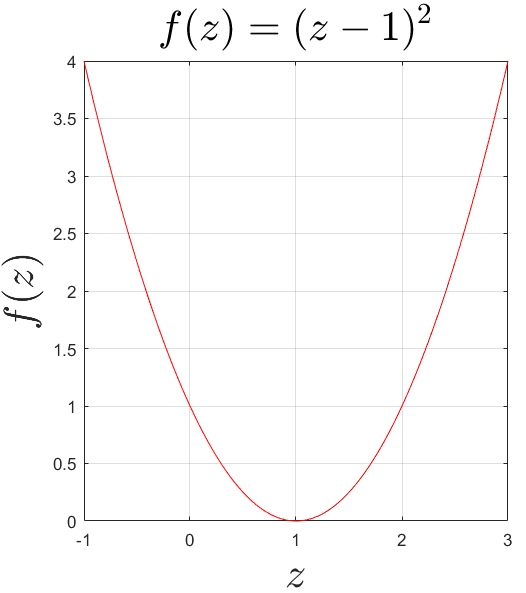}}
	\hspace{0.5 cm}
	\subfloat[Auxiliary Problem]{\label{fig:f}\includegraphics[scale = 0.55]{Plot_4_new.PNG}}
	\caption{Point $ w = 1/3$  is a local minima in (a) and is mapped to the point ${\cal F}(1/3) = -23/27$ which is not a local minimum in (b).}
	\end{figure}
\end{example}
 
{\color{black} Observation~\ref{lm:LocalMinMapping} motivates us to study the local openness of  mappings that appear in widely used optimization problems. A mapping that appears naturally in machine learning optimization problems is matrix multiplication (multiplication of matrices). In the next section, we will focus on characterizing the local openness of this mapping.}

\section{Local Openness of Matrix Multiplication Mappings}\label{Sec:Matrix_Mul}
Motivated by Observation 1, we study the local openness/openness of matrix multiplication mappings. One example that is used in the famous Burer-Monteiro approach for semi-definite programming \cite{burer2005local},  is the symmetric matrix multiplication mapping $\mathcal{M}_{+}: \mathbb{R}^{n \times k} \mapsto \mathcal{R}_{\mathcal{M}_+}$ defined as
\[\mathcal{M}_{+}(\bW) \triangleq \bW{\bW}^{\top},\]
where $\mathcal{R}_{\mathcal{M}_+} \triangleq  \{\, \bZ \in \mathbb{R}^{n \times n} \,\, | \, \bZ \succeq 0, \textrm{ rank}(\bZ) \, \leq \, \min\{n,k\} \}$ is the range of  $\mathcal{M}_{+}$.

\vspace{0.2cm}

Another mapping that is widely used in many optimization problems, such as deep neural networks and matrix completion, is the non-symmetric matrix multiplication mapping $\mathcal{M}: \mathbb{R}^{m \times k} \times \mathbb{R}^{k \times n}  \mapsto \mathcal{R}_{\mathcal{M}}$ defined as
\begin{flalign}
\mathcal{M}({\bW}_1,{\bW}_2) \triangleq {\bW}_1{\bW}_2,
\end{flalign}
where $\mathcal{R}_\mathcal{M}\triangleq  \{\, \bZ \in \mathbb{R}^{m \times n} \,\, | \textrm{ rank}(\bZ) \, \leq \, \min(m,n,k) \}$ is the range of the mapping $\mathcal{M}$.
The matrix multiplication mapping $\mathcal{M}({\bW}_1,{\bW}_2)$ naturally appears in deep models and is widely used as a non-convex factorization for rank constrained problems, see \cite{wang2016unified, bhojanapalli2016global, ge2016matrix, srebro2003weighted, sun2015matrix}. To our knowledge, the complete characterization of the local openness of this mapping has not been studied in the optimization literature before. Similarly, the symmetric matrix multiplication mapping $\mathcal{M}_{+} (\bW)$ is widely used as a non-convex factorization in {\color{black}semi-definite programming (SDP)}, see \cite{burer2005local, zheng2016convergence, park2016non, boumal2016non}, and the characterization of the openness of this mapping remains unsolved. \\

While the classical open mapping theorem in \cite{rudin1973functional} states that  surjective continuous linear operators are open, this is not true for general bilinear mappings such as matrix product. In fact, by providing a simple counterexample of a bilinear mapping that is not open, \cite{horowitz1975elementary} shows that the linear case cannot be generally extended to multilinear maps. Several papers, see \cite{balcerzak2013openness, balcerzak2005multiplying, behrends2011products}, investigate this bilinear mapping and provide a characterization of the points where this  mapping is open. The more general matrix multiplication mapping $\mathcal{M}$ was studied in \cite{behrends2017matrix}. The former paper provides necessary and sufficient conditions under which the mapping is locally open in $\mathbb{R}^{m \times n}$. However, in our framework the (relative) local openness should be studied with respect to the range of the mapping $\mathcal{R}_{\mathcal{M}}$ which can be different from $\mathbb{R}^{m \times n}$ when $k<\min\{m,n\}$.\\

For  ${\bW}_1 \in \mathbb{R}^{m \times k}$ and ${\bW}_2 \in \mathbb{R}^{k \times n}$ with $k \geq \min\{m,n\}$, the range of the mapping ${\cal M}(\bW_1,\bW_2) = \bW_1 \bW_2$ is the entire space $\mathbb{R}^{m \times n}$. In this case, which we refer to as the full rank case, \cite[Theorem~2.5]{behrends2017matrix} provides a complete characterization of the pairs $({\bW}_1,{\bW}_2)$ for which the mapping is locally open. However, when $k < \min\{m,n\}$, which we refer to as the rank-deficient case, the mapping is not locally open in $\mathbb{R}^{m \times n}$, but can still be locally open in $\mathcal{R}_{\mathcal{M}}$. For a simple example, consider $\bbW_1=\left[  1 \quad 2  \right]^{\top}$ and $\bbW_2=\left[  1 \quad 1 \right]$. In this example there does not exist $\tbW_1$, $\tbW_2$ perturbations of $\bbW_1$ and $\bbW_2$ respectively such that $\tbW_1\tbW_2=\widetilde{\bZ}$ when $\widetilde{\bZ}$ is a full rank perturbation of $\bZ=\bbW_1\bbW_2$. However, for any rank $1$ perturbation  $\widetilde{\bZ}$, we can find a perturbed pair $(\tbW_1, \tbW_2)$ such that $\widetilde{\bZ} = \tbW_1\tbW_2$.\\ 


In this section we provide a complete characterization of points $({\bW}_1,{\bW}_2)$ for which the mapping ${\cal M}$ is locally open  when $k < \min\{m,n\}$. Moreover, we show in Theorem~\ref{thm:Loc_Open_general} that the symmetric matrix multiplication $\mathcal{M}_{+}$ is open in its range ${\cal R}_{{\cal M}_{+}}$. The proofs of these theorems can be found in Appendices~\ref{local-open-app} and \ref{PSD-local-open-app}, respectively.
We start by restating the main result in \cite{behrends2017matrix}:

\begin{proposition} \label{Prev_result}
	\cite[Theorem 2.5 Rephrased]{behrends2017matrix} Let $\mathcal{M}(\bW_1,\bW_2) = \bW_1 \bW_2$ denote the matrix multiplication mapping with $\bW_1 \in \mathbb{R}^{m\times k}$ and $\bW_2\in \mathbb{R}^{k\times n}$. Assume $k \geq \min \{m,n\}$.  Then the following statements are equivalent:
	\begin{enumerate}
		\item $\mathcal{M}(\cdot,\cdot)$ is locally open at $(\bbW_1,\bbW_2)$.
		\vspace{0.1cm}
		\item $\begin{cases}
		\exists \, \tbW_1 \in \mathbb{R}^{m\times k} \mbox{ such that } \tbW_1 \bbW_2=\bzero  \mbox{ and } \bbW_1 + \tbW_1 \mbox{ is full row rank}.\\
		\hspace{2.2 in} \mbox{or}\\
		\exists \, \tbW_2 \in \mathbb{R}^{k\times n} \mbox{ such that } \bbW_1\tbW_2=\bzero  \mbox{ and } \bbW_2 + \tbW_2 \mbox{ is full column rank.}
		\end{cases}$
		\item $\dim\big( \mathcal{N}(\bbW_1) \, \cap \, \mathcal{C}(\bbW_2) \big)\leq k-m \, \, \,$ or $\, \, \, n-(\rank(\bbW_2) - \dim\big( \mathcal{N}(\bbW_1) \, \cap \, \mathcal{C}(\bbW_2) \big) \leq k - \rank(\bbW_1)$.
	\end{enumerate}
\end{proposition}
The above proposition provides a checkable condition which completely characterizes the local openness of the mapping $\mathcal{M}$ at different points when the range of the mapping is the entire space. Now, let us state our result that characterizes the local openness of the mapping $\mathcal{M}$ in its range, i.e., when $k< \min\{m,n\}$.
\begin{theorem}\label{thm:Loc_Open}
	Let $\mathcal{M}(\bW_1,\bW_2) = \bW_1 \bW_2$ denote the matrix multiplication mapping with $\bW_1 \in \mathbb{R}^{m\times k}$ and $\bW_2\in \mathbb{R}^{k\times n}$. Assume $k < \min \{m,n\}$.  Then if $\rank(\bbW_1) \neq \rank(\bbW_2)$, $\mathcal{M}(\cdot,\cdot)$ is not locally open at $(\bbW_1,\bbW_2)$. Else, if $\rank(\bbW_1)=\rank(\bbW_2)$, then the following statements are equivalent:
	\begin{enumerate}[i)]
		\item $\exists \, \tbW_1 \in \mathbb{R}^{m\times k} \mbox{ such that } \tbW_1\bbW_2=\bzero  \mbox{ and } \bbW_1 + \tbW_1 \mbox{ is full column rank.}$
		\vspace{0.1cm}
		\item $ \exists \,\tbW_2 \in \mathbb{R}^{k \times n} \mbox{ such that } \bbW_1\tbW_2=\bzero  \mbox{ and } \bbW_2 + \tbW_2 \mbox{ is full row rank.}$
		\vspace{0.1cm}
		\item $\dim\big(\, \mathcal{N}(\bbW_1) \, \cap \, \mathcal{C}(\bbW_2)\big)=0$. 
		\vspace{0.1cm}
		\item$\dim\big( \, \mathcal{N}(\bbW_2^{\top}) \, \cap  \, \mathcal{C}(\bbW_1^{\top})\, \big)=0$.
		\vspace{0.1cm}
		\item $\mathcal{M}(\cdot,\cdot)$ is locally open at $(\bbW_1,\bbW_2)$ in its range ${\cal R}_{\cal M}$.
	\end{enumerate}
\end{theorem}
Note that the proof of Theorem~\ref{thm:Loc_Open}, which can be found in Appendix \ref{local-open-app}, is different than the proof of Proposition~\ref{Prev_result}, as in the former we need to work with the set of low-rank matrices. Besides, the conditions in Theorem~\ref{thm:Loc_Open}  are  different than the ones in Proposition~\ref{Prev_result}. For example, while conditions i) and ii) are equivalent in the  rank-deficient case, they are not equivalent in the full-rank case. Moreover, unlike the full-rank case, the condition $\rank(\bbW_1) = \rank(\bbW_2)$ is necessary for local openness in the former result.\\

\noindent\textbf{How much perturbation is needed?} As previously mentioned, local openness can be described in terms of perturbation analysis. In particular, $\mathcal{M}(\cdot , \cdot)$ is said to be locally open at $({{\bW}_1},{{\bW}_2})$ if for a given $\epsilon > 0$,  there exists $\delta>0$ such that for any $\widetilde{\bZ}=\bZ + {\bR}_{\delta} \in \mathcal{R}_{\mathcal{M}}$  with $\|{\bR}_{\delta}\| \leq \delta$, there exists $\tbW_1$, $\widetilde{\bW}_2$ with $\|\tbW_1\|\leq \epsilon$, $\|\widetilde{\bW}_2\| \leq \epsilon$ such that $\widetilde{\bZ}=(\bW_1+\tbW_1)({\bW}_2+\widetilde{\bW}_2)$. {\color{black} Given $\epsilon>0$, we show that for any locally open $({{\bW}_1},{{\bW}_2})$ we need to choose $\delta = \Theta (\epsilon)$.} The details of our analysis can be found in the proof of Theorem \ref{thm:Loc_Open} in Appendix~\ref{local-open-app}

Now we state our result for the mapping $\mathcal{M}_{+}$.
\begin{theorem}\label{thm:Loc_Open_general} Let $\mathcal{M}_+(\bW) = \bW\bW^{\top}$ be the symmetric matrix multiplication mapping. Then $\mathcal{M}_{+}(\cdot)$ is  open in its range $\mathcal{R}_{\mathcal{M}_+}$. 
\end{theorem} 
\noindent\textbf{How much perturbation is needed?} A perturbation bound for the symmetric matrix multiplication was also derived, details in Appendix~\ref{PSD-local-open-app}. {\color{black} Specifically, given an $\epsilon>0$, we show that for any $\bW$ the chosen $\delta$ is of order $\epsilon$, i.e.,  $\delta = \Theta(\epsilon)$.}

\begin{remark} Since the mapping $\mathcal{M}_{+}$ is open in ${\cal R}_{{\cal M}_{+}}$, then by Observation \ref{lm:LocalMinMapping}, any local minimum of the optimization problem
	$\displaystyle{\min_{\bW \in \mathcal{W}}} \;\; \ell(\bW\bW^{\top})
	$ leads to a  local minimum in the optimization problem
	$ \displaystyle{\min_{\bZ \in \mathcal{Z}}} \;\; \ell(\bZ)$ where $\mathcal{Z} = \{\bZ \;\mid \; \bZ \succeq 0, \;\bZ = \bW\bW^{\top}, \;\bW \in \mathcal{W}\}$. {\color{black} Consequently, if every local minimum of the optimization problem on $\bZ$ is globally optimal, then every local minimum of the first optimization problem is global.} This provides a simple and intuitive proof for  the Burrer-Monteiro result \cite[Proposition~2.3]{burer2005local}; moreover, it extends it by relaxing  the continuity assumption on $\ell(\cdot)$.
\end{remark}

\begin{remark} It follows from Theorem \ref{thm:Loc_Open} that when ${{\bW}_1}$ is full column rank and ${{\bW}_2}$ is full row rank, the mapping $\mathcal{M}(\cdot, \cdot)$ is locally open at $({{\bW}_1},{{\bW}_2})$. This special case of our result was observed in other works; see, e.g.,  \cite[Proposition~4.2]{sun2015matrix}. 
\end{remark}

In the next sections, we use our local openness result to characterize the cases where the local optima of various training optimization problems of the form \eqref{Prob_general} are globally optimal. 

\section{Non-linear Deep Neural Network with a Pyramidal Structure}\label{sec:non-linear-local-global}
{\color{black} In this section, we utilize Observation~\ref{lm:LocalMinMapping} to study the local/global optima equivalence in non-linear deep neural networks having a specific pyramidal structure. Towards that end,} consider the non-linear deep neural network optimization problem with a pyramidal structure
\begin{equation}\label{Non_Linear_SE}
\min_{\bW} \, \, \ell \big(\mathcal{F}_h(\bW)\big)
\end{equation}
with
\[\mathcal{F}_i(\bW) \triangleq {\bsigma}_i \big({\bW}_{i}\mathcal{F}_{i-1}(\bW)\big), \textrm{ for } i \in  \{2,\ldots,h\}, \quad \mbox{and} \quad \mathcal{F}_1(\bW) \triangleq {\bsigma}_1({\bW}_1\bX)
\]
where ${\bsigma}_i(\cdot):\mathbb{R}\mapsto \mathbb{R}$ is a continuous and strictly monotone activation function applied component-wise to the entries of each layer, i.e., ${\bsigma}_i(\bA) = [{\bsigma}_i({\bA}_{jk})]_{j,k}$. Here $\bW=\big( {\bW}_i \big)_{i=1}^h$ where $\, \, {\bW}_i \in \mathbb{R}^{d_i \times d_{i-1}}$ is the weight matrix of layer $i$, and $\bX \in \mathbb{R}^{d_0 \times n}$ is the input training data. In this section, we consider the pyramidal network structure with  $d_0 >n$ and  $d_i \leq d_{i-1}$ for $1 \leq i\leq h$; see \cite{nguyen2017loss} for more details on these types of networks. In their paper, \cite{nguyen2017loss} show that every critical point $\bar{\bW}$ of problem (\ref{Non_Linear_SE}) with $\bar{\bW}_i$'s being full row rank is a global optimum when both ${\bsigma} (\cdot)$ and $\ell (\cdot)$ are differentiable. In this section, we relax the differentiability assumption on both the activation and loss functions and  {\color{black} establish local/global optima equivalence for problem~\eqref{Non_Linear_SE}.}\\ 

{\color{black} We first show that when $\bX$ is full column rank and the functions ${\bsigma}_i$'s are all continuous and strictly monotone, the image set of the mapping $\mathcal{F}_h$ is convex. First notice that the full rankness of $\bX$ and continuity and strict monotinicity of $\bsigma_1$ {\color{black}yield} a range of $\mathbb{R}^{d_1 \times n}$ for ${\cal F}_1(\bW) = \bsigma(\bW_1 \bX)$. By the same reasoning, we get that the range of the mapping ${\cal F}_i(\bW)$ is $\mathbb{R}^{d_i \times n}$. Hence, the image set of the mapping $\mathcal{F}_h$ is the full space $\mathbb{R}^{d_n \times n}$ which is clearly convex.} {\color{black}Furthermore, if $\ell(\cdot)$ is convex,} then every local optimum of the corresponding auxiliary optimization problem~\eqref{Prob_general_formulated} is global. We now show that when ${\bW}_i$'s are all full row rank and  the functions ${\bsigma}_i$'s are all strictly monotone, the mapping $\mathcal{F}_h$ is  locally open at $\bW$.
\begin{lemma}\label{lem:Pyramidal}
	Assume the functions ${\bsigma}_i(\cdot):\mathbb{R}\mapsto \mathbb{R}$ are all continuous and strictly monotone. If $\bar{\bW}_i$'s are all full row rank, then the mapping $\mathcal{F}_h$ defined in (\ref{Non_Linear_SE}) is locally open at the point $\bar{\bW} = (\bar{\bW}_1,\ldots, \bar{\bW}_h)$.
\end{lemma}
Before proving this result, we would like to remark that many of the popular activation functions such as logit, tangent hyperbolic, and leaky ReLU are strictly monotone and satisfy the assumptions of this lemma. {\color{black} Our result shows that using such activation functions on pyramidal neural networks yield underlying landscapes with all local minima having full-rank weight parameters being global.} {\color{black} Our result does not hold when using ReLU activation function as it is not an open mapping.}
\begin{proof}
	We prove the result by means of induction. {\color{black} First notice that ${\cal F}_1$ can be seen as the composition of the $\bsigma_1(\cdot)$ and the linear mapping $\bar{\bW} \, \bX$. Using the strict monotonicity assumption, we get that $\bsigma_1(\cdot)$ is invertible. By the definition of open mappings, it follows that $\bsigma_1(\cdot)$ is an open mapping. Hence, using the local openness property of linear maps and the openness of ${\bsigma}_1 (\cdot)$, the composition property of open maps detailed in Property~\ref{prop:composition} directly implies that ${\cal F}_1$ is open.} Now assume $\mathcal{F}_{k-1}\Big(\big(\bar{\bW}_i\big)_{i=1}^{k-1}\Big)$ is locally open at $\big(\bar{\bW}_i \big)_{i=1}^{k-1}$, then using Proposition \ref{Prev_result} for a full rank $\bar{\bW}_k$, the mapping $\bar{\bW}_k\mathcal{F}_{k-1}\Big(\big(\bar{\bW}_i \big)_{i=1}^{k-1}\Big)$ is locally open at $\big(\bar{\bW}_k, (\bar{\bW}_i)_{i=1}^{k-1}\big)$. Finally, the composition property of open maps and strict monotonicity of ${\bsigma}_k (\cdot)$ imply that $\mathcal{F}_k\Big((\bar{\bW}_i)_{i=1}^{k}\Big)$ is locally open at $\big(\bar{\bW}_i \big)_{i=1}^k$.
\end{proof}

Lemma~\ref{lem:Pyramidal} in conjunction with  Observation \ref{lm:LocalMinMapping} implies that if ${\bbW}$ is a local optimum of problem (\ref{Non_Linear_SE}) with ${\bbW}_i$'s being full row rank, then $\bar{\bZ}=\mathcal{F}_h(\bbW)$ is a local optimum of the corresponding auxiliary problem 
$\displaystyle{\min_{\bZ \in {\cal Z}}} \,\, \ell(\bZ)
$ where ${\cal Z}$ is a convex set.  Consequently, $\bar{\bZ}$ is a global optimum of problem (\ref{Non_Linear_SE}) when the loss function $\ell(\cdot)$ is convex. {\color{black} Given an oracle that returns a local minima of problem~\eqref{Non_Linear_SE}, one can check whether ${\bbW}_i$'s are full row rank. If true, we guarantee that the point is globally optimal.} Unlike the results in~\cite{nguyen2017loss}, our lemma holds for non-differentiable activation and loss functions. A popular activation function that is strictly monotone and not differentiable is the Leaky ReLU, for which our result holds.

\section{Two-Layer Linear Neural Network}\label{sec:two-layer-results}
Consider the two layer linear neural network optimization problem
\begin{equation}\label{Two-layer-SE-loss-2}
\min_{\bW} \, \,\,\, \dfrac{1}{2} \|{\bW}_2{\bW}_1\bX \, - \, \bY \, \|^2,
\end{equation}
where ${\bW}_2 \in  \mathbb{R}^{d_2 \times d_1}$,  and ${\bW}_1 \in \mathbb{R}^{d_1 \times d_0}$ are weight matrices, $\bX \in \mathbb{R}^{d_0 \times n}$ is the input data, and $\bY \in \mathbb{R}^{d_2 \times n}$ is the target training data. Using our transformation, the corresponding auxiliary optimization problem can be written as
\begin{equation}\label{Two-layer-SE-loss-formulated}
\begin{array}{ll}
\min_{\bZ}  \,\; \dfrac{1}{2}||\bZ \bX- \bY||^2 \quad\quad \st \quad \rank(\bZ) \leq \min\{d_2, d_1, d_0\}
\end{array}.
\end{equation}

\cite[Theorem~2.3]{kawaguchi2016deep} shows that when $\bX {\bX}^{\top}$ and $\bY {\bX}^{\top}$ are full rank, $d_2 \leq d_0$, and when $\bY {\bX}^{\top}(\bX {\bX}^{\top})^{-1}\bX {\bY}^{\top}$ has $d_2$ distinct eigenvalues, every local optimum of {\color{black} \eqref{Two-layer-SE-loss-2}} is global and all saddle points are second order saddles. While the local/global equivalence result holds for deeper networks, the property that all saddles are second order does not hold in that case. Another result by
\cite[Theorem~2.2]{yun2017global} show that when $\bX {\bX}^{\top}$, $\bY {\bX}^{\top}$, and $\bY {\bX}^{\top}(\bX {\bX}^{\top})^{-1}\bX {\bY}^{\top}$ are full rank, every local optimum {\color{black} of \eqref{Two-layer-SE-loss-2}} is global. 
In this section, without any assumptions on both $\bX$ and $\bY$, we use local openness to show that the latter result holds for $2$-layer linear networks. Moreover, we show that every degenerate local optima is global even when replacing the square loss error by a general convex loss function, see Corollary \ref{cor: degenerate2layer}. {\color{black} We start by relaxing the full rankness assumption on $\bX$.

	\begin{lemma}\label{lm:Relax-X-Assumption}
		Every local minimum of problem (\ref{Two-layer-SE-loss-formulated}) is global.
	\end{lemma}
	\begin{proof}
		Let $ r_{\bX} = \rank(\bX)$ and $\bU_{\bX}{\bSigma}_{\bX}\bV_{\bX}^{\top}$ with $\bU_{\bX} \in \mathbb{R}^{d_0 \times d_0}$, $\bSigma_{\bX} \in \mathbb{R}^{d_0 \times n}$, and $\bV_{\bX} \in \mathbb{R}^{n \times n}$ be a singular value decomposition of $\bX$.{\color{black} Then
		\begin{flalign*}
		\| \bZ\bX - \bY \|^2 & =  \|\bZ \bU_{\bX}\bSigma_{\bX}\bV_{\bX}^{\top} - \bY\|^2\\
		& = \|\left(\bZ \bU_{\bX}\bSigma_{\bX}\bV_{\bX}^{\top} - \bY\right) \bV_{\bX}\|^2\\
		& = \| \bZ \bU_{\bX}\begin{array}{cc}\left[\big(\bSigma_{\bX}\big)_{:,1:r_{\bX}} \big| \bzero\right]\end{array} -\bY\bV_{\bX} \|^2\\
		& = \| \bZ\bU_{\bX}\big(\bSigma_{\bX}\big)_{:,1:r_{\bX}} -\big(\bY\bV_{\bX}\big)_{:,1:r_{\bX}} \|^2 +\underbrace{\|\big(\bY\bV_{\bX}\big)_{:, r_{\bX} +1: n} \|^2}_{\mbox{constant in problem (\ref{Two-layer-SE-loss-formulated})}}.
		\end{flalign*}

where the second and third equalities hold since $\bV_{\bX}^{\top}\bV_{\bX} = \bI$. }Since $\bU_{\bX}\big(\bSigma_{\bX}\big)_{:,1:r_{\bX}}$ is full column rank, then the linear mapping $\bZ\bU_{\bX}\big(\bSigma_{\bX}\big)_{:,1:r_{\bX}}$ is open, and 
{\color{black}\[\rank(\bZ\bU_{\bX}\big(\bSigma_{\bX}\big)_{:,1:r_{\bX}}) \leq \min\{\rank(\bZ), r_{\bX}\} \leq \min\{d_2, d_1, d_0, r_{\bX}\}.\]}
Consequently, every local minimum of (\ref{Two-layer-SE-loss-formulated}) corresponds to a local minimum in problem 
\begin{equation}
\begin{array}{ll}
\min_{\bar{\bZ} \in \mathbb{R}^{d_2 \times r_{\bX}}}  \;\, \dfrac{1}{2}\|\bar{\bZ} - \bar{\bY}\|^2\quad 
\st  \;\; \rank(\bar{\bZ}) \leq \min\{d_2, d_1, d_0, r_{\bX}\}
\end{array},
\end{equation}
where $\bar{\bY} = \big(\bY\bV_{\bX}\big)_{:,1:r_{\bX}}$. The result follows using \cite[Theorem~2.2]{lu2017depth}.
\end{proof}

We next state the main results for problem~\eqref{Two-layer-SE-loss-2}.}

\begin{theorem}\label{thm: Two-layer}
	Every local minimum of problem (\ref{Two-layer-SE-loss-2}) is global. Moreover, every degenerate saddle point of problem (\ref{Two-layer-SE-loss-2}) is a second order saddle.
\end{theorem}
\begin{proof}
	{\color{black} A detailed proof of the theorem is relegated to Appendix~\ref{two-layer-app}. We now present a sketch of the proof. We separately considered the cases of degenerate and non-degenerate critical points.	For the former case, we construct a descent direction for critical points that are not global. Such directions are constructed using the null-space of the rank deficient matrices, and can be of independent interest when developing algorithms for training deep neural networks. For the non-degenerate case, it follows by Theorem~\ref{thm:Loc_Open_general} that the matrix product is locally open at a given non-degenerate local minimum. Then by Observation~\ref{lm:LocalMinMapping}, the local minimum can be mapped to a local minimum of problem~\eqref{Two-layer-SE-loss-formulated} which is globally optimal by Lemma~\ref{lm:Relax-X-Assumption}. 
}  
\end{proof}

\begin{corollary}\label{cor: degenerate2layer}
	Let the square loss error in  (\ref{Two-layer-SE-loss-2}) be replaced by a general convex loss function $\ell(\cdot)$. Then every degenerate critical point is either a global minimum or a second order saddle.
\end{corollary}
\begin{proof}
{\color{black}The proof of the corollary is relegated to Appendix~\ref{App:corollary_degenerate2layer}}
\end{proof}

\cite{baldi1989neural} and \cite{srebro2003weighted} show the same result when both $\bX$ and $\bY$ are full row rank. Theorem \ref{thm: Two-layer} generalizes their results by relaxing the assumptions on both $\bX$ and $\bY$.

\section{Multi-Layer Linear Neural Network}\label{sec:multi-layer-results}
Consider the training problem of multi-layer deep linear neural networks: 
\begin{equation}\label{Linear_SEloss-2}
\min_{\bW} \, \, \dfrac{1}{2} \|{\bW}_h \cdots {\bW}_1\bX \, - \, \bY \, \|^2.
\end{equation}
Here  ${\bW}=\big( {\bW}_i \big)_{i=1}^h$, ${\bW}_i \in \mathbb{R}^{d_i \times d_{i-1}}$ are the weight matrices, $\bX \in \mathbb{R}^{d_0 \times n}$ is the input training data, and $\bY \in \mathbb{R}^{d_h \times n}$ is the target training data. Based on our general framework, the corresponding auxiliary optimization problem is given by 
\begin{equation} \label{Transformed_Linear_SEloss-2}
\begin{array}{ll}
\min_{\bZ \in \mathbb{R}^{d_h \times n}}  \dfrac{1}{2} ||\bZ\bX \, - \, \bY||^2 \quad
\st \;\; \rank(\bZ) \leq d_p \triangleq \min_{0 \leq i\leq h} \,d_i.
\end{array}
\end{equation}

\cite{lu2017depth} showed that when $\bX$ and $\bY$ are full row rank, every local minimum of \eqref{Linear_SEloss-2} is global. {\color{black} We show that the full rankness assumption on $\bY$ cannot be simply relaxed by providing the following counterexample:}
\[\bX=\bI \quad \bbW_3 = \left[\begin{array}{c} 1 \\ 0 \end{array} \right], \quad \bbW_2=\left[\, 0 \,  \right], 
\quad \bbW_1=\left[\begin{array}{c c} 1 & 0\end{array}\right], \quad \bY=\left[\begin{array}{c c} 0 & 0\\ 0 & 1\\ \end{array}\right].\]
In this example, the point $\bbW=(\bbW_1,\bbW_2,\bbW_3)$ is a local optimum of a $3$-layer deep linear model that is not global. Despite this counterexample, we will provide a set of conditions on the network architecture for which all local minima of \eqref{Linear_SEloss-2} are global even if $\bY$ is not full rank. Before proceeding to the proof we define the mapping
\[\mathcal{M}_{i,j}({\bW}_i, \ldots , {\bW}_j) : \{{\bW}_i, \ldots , {\bW}_j\} \rightarrow \mathcal{R}_{\mathcal{M}_{i,j}} \quad \mbox{for } i>j,\]
where $\mathcal{R}_{\mathcal{M}_{i,j}} \triangleq \{\bZ={\bW}_i \cdots {\bW}_j \in \mathbb{R}^{d_i \times d_{j-1}} \, | \, \rank(\bZ) \leq \min_{j-1 \leq l \leq i}d_l \, \}$. {\color{black} We start by re-stating Theorem $3.1$ of \cite{lu2017depth} using our notation. 
\begin{lemma}\label{lm: Kenji's Result}
	If $\bW$ is non-degenerate, then $\mathcal{M}_{h,1}({\bW})={\bW}_h \cdots {\bW}_1$ is locally open at ${\bW}$.
\end{lemma}

\begin{proof}
	We construct a proof by induction on $h$ to show the desired result. When $h=2$, we either have $d_1 < \min \{d_2,d_0\}$ or $d_1 \geq \min \{d_2, d_0\}$. In the first case,  
	\[d_1=\rank({\bW}_2{\bW}_{1}) \leq  \rank({\bW}_{1}) \leq d_1 \Rightarrow\rank({\bW}_{1}) = d_1 , \] 
	and 
	\[d_1=\rank({\bW}_2{\bW}_{1}) \leq  \rank({\bW}_{2}) \leq d_1 \Rightarrow\rank({\bW}_{2}) = d_1.\]
	Since $\bW_1$ is full row rank and $\bW_2$ is full column rank, then by Theorem \ref{thm:Loc_Open}, {\color{black}choosing $\widetilde{\bW}_1$ = $\widetilde{\bW}_2=\bzero$} yields $\mathcal{M}_{2,1}(\cdot)$ is locally open at $({\bW}_2,{\bW}_1)$.
	In the second case, either
	\[d_2=\rank({\bW}_2{\bW}_{1}) \leq  \rank({\bW}_{2}) \leq d_2 \Rightarrow \rank({\bW}_{2})  = d_2,\]
	or 
	\[d_0=\rank(\bW_2\bW_1) \leq  \rank(\bW_1) \leq d_0 \Rightarrow \rank(\bW_1)  = d_0.\]
	Thus, either $\bW_2$ is full row rank or $\bW_1$ is full column rank, then by Proposition \ref{Prev_result}, $\mathcal{M}_{2,1}(\cdot)$ is locally open at $({\bW}_2,{\bW}_1)$. Now assume the result holds for the product of $h$ matrices $\mathcal{M}_{h,1}({\bW})$, we show it is true for $\mathcal{M}_{h+1,1}({\bW})$. Since 
	\[d_p=\rank({\bW}_h \ldots  {\bW}_{1}) \leq  \rank({\bW}_{p+1}{\bW}_p) \leq d_p \Rightarrow \rank({\bW}_{p+1}{\bW}_p) = d_p,\]
	then using Proposition \ref{Prev_result}, we get $\mathcal{M}_{p+1,p}(\cdot)$ is locally open at $({\bW}_{p+1},{\bW}_p)$. So we can replace ${\bW}_{p+1}{\bW}_{p}$ by a new matrix ${\bZ}_p$ with rank $d_p$. Then by induction hypothesis, the product mapping ${\cal M}_{h+1,1}({\bW}) = {\bW}_{h+1}\cdots {\bW}_{p+2}{\bZ}_p{\bW}_{p-1}\cdots {\bW}_1$ is locally open at ${\bW}$. Since the composition of locally open maps is locally open, the result follows.
	\end{proof}

We next show that under a set of necessary conditions, every local minimum of problem (\ref{Linear_SEloss-2}) is global. 

\begin{lemma}\label{lm:Multiplie-layer-Non-Deg}
	Every non-degenerate local minimum of (\ref{Linear_SEloss-2}) is global minimum.
\end{lemma}
{\color{black}\begin{proof}
	 Suppose $\bW=({\bW}_h, \ldots ,{\bW}_1)$ is a non-degenerate local minimum. Then it follows by Lemma \ref{lm: Kenji's Result} that $\mathcal{M}_{h,1}$ is locally open at ${\bW}$. Then by Observation \ref{lm:LocalMinMapping}, $\bZ = \mathcal{M}_{h}({\bW}_h, \ldots ,{\bW}_1)$ is a local optimum of problem (\ref{Transformed_Linear_SEloss-2}) which is in fact global by Lemma \ref{lm:Relax-X-Assumption}.
\end{proof}



As previously mentioned, due to a simple counterexample, we cannot in general relax the full rankness assumption on $\bY$. We now determine {\color{black} problem} structures for which every degenerate local minimum is global, i.e., (due to Lemma 4) problem structures for which every local minimum is global.

\begin{theorem}\label{thm:Deep-linear-Result}
If there does not exist $p_1$ and $p_2$, $1 \leq p_1 < p_2 \leq h-1$ with $d_h > d_{p_2}$ and $d_0 > d_{p_1}$, then every local minimum of problem~\eqref{Linear_SEloss-2} is a global minimum.
\end{theorem}

\begin{proof}
	The proof of the theorem is relegated to Appendix~\ref{Deep-linear-result-app}. 	
\end{proof}

\begin{remark}Following the same steps of the proof of Theorem \ref{thm:Deep-linear-Result}, we get the same result when replacing the square error loss by a general convex and differentiable function $\ell(\cdot)$. Moreover, if the range of the mapping $\mathcal{M}_{h}$ is the entire space, i.e., $\min_{0\leq i \leq h} \, d_i = \min\{d_h,d_0\}$, the auxiliary problem (\ref{Transformed_Linear_SEloss-2}) is unconstrained and convex. Then, as we show in Corollary \ref{Corollary: Deep-Network-whole-space}, every non-degenerate critical point is global, and every degenerate critical point is either a saddle point or a global minimum; which generalizes \cite[Theorem~2.1]{yun2017global}.
\end{remark}

\begin{remark}
{\color{black} Practically, Theorem \ref{thm:Deep-linear-Result} provides a simple test that uses the network structure to determine whether every local minimum of the underlying landscape is global. }

\end{remark}

\begin{corollary} \label{Corollary: Deep-Network-whole-space}
	Consider problem (\ref{Linear_SEloss-2}) with general convex and differentiable loss function $\ell(\cdot)$. When $\min_i d_i = \min(d_h,d_0)$, every non-degenerate critical point is global, and every degenerate critical point is either a saddle point or a global minimum.
\end{corollary}
\begin{proof}
	Suppose $\bbW$ is a degenerate critical point, then by replacing the square loss error by a general convex and differentiable function $\ell(\cdot)$ in Theorem \ref{thm:Deep-linear-Result}, we get that $\bbW$ is either a saddle or a global minimum. Suppose $\bbW=(\bbW_h, \ldots ,\bbW_1)$ is a non-degenerate critical point and $k \triangleq  \min_i \, d_i = \min (d_h,d_0)$, we follow the same steps of the proof of \cite[Theorem~2.1]{yun2017global} to show the desired result. First note that
	\[\dfrac{\partial \ell(\bW_h\cdots \bW_1\bX)}{\partial \bW_1}\Bigg|_{\bW=\bbW}=\bbW_2^{\top}\cdots \bbW_h^{\top}\nabla \ell(\bbW_h\cdots \bbW_1\bX){\bX}^{\top},\]
	and
	$\dfrac{\partial \ell(\bW_h\cdots \bW_1\bX)}{\partial \bW_h}\Bigg|_{\bW=\bbW}=\nabla \ell(\bbW_h\cdots \bbW_1\bX){\bX}^{\top}\bbW_1^{\top}\cdots \bbW_{h-1}^{\top},$
	where $\nabla \ell$ is the gradient mapping of the function $\ell(\cdot)$. If $k=d_h$, let $\bS=\bbW_2^{\top}\cdots \bbW_h^{\top} \in \mathbb{R}^{d_1 \times k}$ and $\bT=\nabla \ell(\bbW_h\cdots \bbW_1\bX){\bX}^{\top}$. It follows that
	\[k = \rank(\bbW_h \cdots \bbW_1) \leq \rank(\bS^{\top}) \leq k \Rightarrow \rank(\bS)=k.\]
	Since $\bbW$ is a critical point and ${\bS}^{\top}$ is full row rank, we get
	\[
	0=\Bigg\| \dfrac{\partial \ell(\bW_h\cdots \bW_1\bX)}{\partial {\bW}_1}\Bigg|_{\bW=\bbW} \Bigg\|^2=tr\big(\bT^{\top}\bS^{\top}\bS\bT\big) \geq {\sigma}_{min}^2(\bS)\|\bT\|^2.
	\]
	Thus, $ \bT = \nabla \ell(\bbW_h \cdots \bbW_1\bX){\bX}^{\top}=\bzero,$
	which by convexity $\ell(\cdot)$ implies that $\bbW$ is a global minimum. Similarly, we can show that the case of $k=d_0$ results in the global optimality of $\bbW$ as well.
	\end{proof}
	
\section{Conclusion}\label{Sec:Con}
{\color{black}In this paper, we develop a unifying landscape analysis framework for studying the local/global equivalence for several non-convex objective functions that arise in statistical machine learning settings. In particular, our proposed framework utilizes
the concept of local openness from differential geometry to provide sufficient conditions under which local optima of the objective function are global. While we narrow down our focus to a certain class of non-convex problems, the studied class is general enough to cover many practical applications such as matrix completion,
low-rank matrix recovery, and deep learning. In our work, we completely characterize the local openness of matrix multiplication mapping in its range. More specifically, we provide necessary and sufficient conditions under which the matrix multiplication mapping is locally open. Based on this theoretical result, we develop a complete characterization
of the local/global optima equivalence of multi-layer linear neural networks and provide sufficient conditions for which no spurious local optima exist under hierarchical non-linear deep neural networks. Unlike many existing results that focus on a particular algorithm (example gradient descent) and specific input data distribution, our result depends on the {\color{black} network structure} and does not rely on the probability distribution of the input data.\\

Leveraging on our
results, \cite{zhu2019distributed} show that every second-order stationary point of the low-rank matrix factorization problem is global. Our framework was also used to show similar results for shallow linear neural networks~\cite{zhu2019global}, low-rank matrix recovery~\cite{li2020global1}, and meta learning objectives on several reinforcement learning tasks~\cite{molybog2020global}. Such favorable geometry directly implies that local search methods that compute second-order stationary solutions will converge to global minima of the objective functions.}

\section {Acknowledgement}
The authors would like to thank Li Zhang for pointing out a mistake in the proof of Theorem~12. 

\bibliographystyle{plain}
\bibliography{references}

\begin{thebibliography}{10}

\bibitem{allen2018learning}
Zeyuan Allen-Zhu, Yuanzhi Li, and Yingyu Liang.
\newblock Learning and generalization in overparameterized neural networks,
  going beyond two layers.
\newblock {\em arXiv preprint arXiv:1811.04918}, 2018.

\bibitem{allen2018convergence}
Zeyuan Allen-Zhu, Yuanzhi Li, and Zhao Song.
\newblock A convergence theory for deep learning via over-parameterization.
\newblock {\em arXiv preprint arXiv:1811.03962}, 2018.

\bibitem{arora2018convergence}
Sanjeev Arora, Nadav Cohen, Noah Golowich, and Wei Hu.
\newblock A convergence analysis of gradient descent for deep linear neural
  networks.
\newblock {\em arXiv preprint arXiv:1810.02281}, 2018.

\bibitem{arora2019exact}
Sanjeev Arora, Simon~S Du, Wei Hu, Zhiyuan Li, Ruslan Salakhutdinov, and
  Ruosong Wang.
\newblock On exact computation with an infinitely wide neural net.
\newblock {\em arXiv preprint arXiv:1904.11955}, 2019.

\bibitem{balcerzak2013openness}
M.~Balcerzak, A.~Majchrzycki, and A.~Wachowicz.
\newblock Openness of multiplication in some function spaces.
\newblock {\em Taiwanese J. Math}, 17:1115--1126, 2013.

\bibitem{balcerzak2005multiplying}
M.~Balcerzak, A.~Wachowicz, and W.~Wilczy{\'n}ski.
\newblock Multiplying balls in the space of continuous functions on $[0, 1] $.
\newblock {\em Studia Mathematica}, 170:203--209, 2005.

\bibitem{baldi1989neural}
P.~Baldi and K.~Hornik.
\newblock Neural networks and principal component analysis: Learning from
  examples without local minima.
\newblock {\em Neural networks}, 2(1):53--58, 1989.

\bibitem{behrends2011products}
E.~Behrends.
\newblock Products of $ n $ open subsets in the space of continuous functions
  on $[0, 1] $.
\newblock {\em Studia Mathematica}, 204:73--95, 2011.

\bibitem{behrends2017matrix}
E.~Behrends.
\newblock Where is matrix multiplication locally open?
\newblock {\em Linear Algebra and its Applications}, 517:167--176, 2017.

\bibitem{bhojanapalli2016global}
S.~Bhojanapalli, B.~Neyshabur, and N.~Srebro.
\newblock Global optimality of local search for low rank matrix recovery.
\newblock In {\em Advances in Neural Information Processing Systems}, pages
  3873--3881, 2016.

\bibitem{blum1989training}
A.~Blum and R.~L. Rivest.
\newblock Training a 3-node neural network is np-complete.
\newblock In {\em Advances in neural information processing systems}, pages
  494--501, 1989.

\bibitem{boumal2016non}
N.~Boumal, V.~Voroninski, and A.~Bandeira.
\newblock The non-convex burer-monteiro approach works on smooth semidefinite
  programs.
\newblock In {\em Advances in Neural Information Processing Systems}, pages
  2757--2765, 2016.

\bibitem{burer2005local}
S.~Burer and R.~D.~C. Monteiro.
\newblock Local minima and convergence in low-rank semidefinite programming.
\newblock {\em Mathematical Programming}, 103(3):427--444, 2005.

\bibitem{choromanska2015loss}
A.~Choromanska, M.~Henaff, M.~Mathieu, G.~B. Arous, and Y.~LeCun.
\newblock The loss surfaces of multilayer networks.
\newblock In {\em Artificial Intelligence and Statistics}, pages 192--204,
  2015.

\bibitem{ding2019sub}
Tian Ding, Dawei Li, and Ruoyu Sun.
\newblock Sub-optimal local minima exist for almost all over-parameterized
  neural networks.
\newblock {\em arXiv preprint arXiv:1911.01413}, 2019.

\bibitem{du2018power}
Simon~S Du and Jason~D Lee.
\newblock On the power of over-parametrization in neural networks with
  quadratic activation.
\newblock {\em arXiv preprint arXiv:1803.01206}, 2018.

\bibitem{du2018gradient}
Simon~S Du, Xiyu Zhai, Barnabas Poczos, and Aarti Singh.
\newblock Gradient descent provably optimizes over-parameterized neural
  networks.
\newblock {\em arXiv preprint arXiv:1810.02054}, 2018.

\bibitem{freeman2016topology}
C~Daniel Freeman and Joan Bruna.
\newblock Topology and geometry of half-rectified network optimization.
\newblock {\em arXiv preprint arXiv:1611.01540}, 2016.

\bibitem{ge2016matrix}
R.~Ge, J.~D. Lee, and T.~Ma.
\newblock Matrix completion has no spurious local minimum.
\newblock In {\em Advances in Neural Information Processing Systems}, pages
  2973--2981, 2016.

\bibitem{hardt2016identity}
Moritz Hardt and Tengyu Ma.
\newblock Identity matters in deep learning.
\newblock {\em arXiv preprint arXiv:1611.04231}, 2016.

\bibitem{horowitz1975elementary}
C.~Horowitz.
\newblock An elementary counterexample to the open mapping principle for
  bilinear maps.
\newblock {\em Proceedings of the American Mathematical Society},
  53(2):293--294, 1975.

\bibitem{kawaguchi2016deep}
K.~Kawaguchi.
\newblock Deep learning without poor local minima.
\newblock In {\em Advances in Neural Information Processing Systems}, pages
  586--594, 2016.

\bibitem{laurent2018deep}
Thomas Laurent and James Brecht.
\newblock Deep linear networks with arbitrary loss: All local minima are
  global.
\newblock In {\em International conference on machine learning}, pages
  2902--2907. PMLR, 2018.

\bibitem{li2018over}
Dawei Li, Tian Ding, and Ruoyu Sun.
\newblock Over-parameterized deep neural networks have no strict local minima
  for any continuous activations.
\newblock {\em arXiv preprint arXiv:1812.11039}, 2018.

\bibitem{li2020global1}
Shuang Li, Qiuwei Li, Zhihui Zhu, Gongguo Tang, and Michael~B Wakin.
\newblock The global geometry of centralized and distributed low-rank matrix
  recovery without regularization.
\newblock {\em arXiv preprint arXiv:2003.10981}, 2020.

\bibitem{li2018learning}
Yuanzhi Li and Yingyu Liang.
\newblock Learning overparameterized neural networks via stochastic gradient
  descent on structured data.
\newblock In {\em Advances in Neural Information Processing Systems}, pages
  8157--8166, 2018.

\bibitem{lu2017depth}
H.~Lu and K.~Kawaguchi.
\newblock Depth creates no bad local minima.
\newblock {\em arXiv preprint arXiv:1702.08580}, 2017.

\bibitem{molybog2020global}
Igor Molybog and Javad Lavaei.
\newblock Global convergence of maml for lqr.
\newblock {\em arXiv preprint arXiv:2006.00453}, 2020.

\bibitem{nguyen2017loss}
Q.~Nguyen and M.~Hein.
\newblock The loss surface of deep and wide neural networks.
\newblock {\em arXiv preprint arXiv:1704.08045}, 2017.

\bibitem{nguyen2019connected}
Quynh Nguyen.
\newblock On connected sublevel sets in deep learning.
\newblock In {\em International Conference on Machine Learning}, pages
  4790--4799. PMLR, 2019.

\bibitem{nguyen2018loss}
Quynh Nguyen, Mahesh~Chandra Mukkamala, and Matthias Hein.
\newblock On the loss landscape of a class of deep neural networks with no bad
  local valleys.
\newblock {\em arXiv preprint arXiv:1809.10749}, 2018.

\bibitem{oymak2019towards}
Samet Oymak and Mahdi Soltanolkotabi.
\newblock Towards moderate overparameterization: global convergence guarantees
  for training shallow neural networks.
\newblock {\em arXiv preprint arXiv:1902.04674}, 2019.

\bibitem{park2016non}
D.~Park, A.~Kyrillidis, C.~Caramanis, and S.~Sanghavi.
\newblock Non-square matrix sensing without spurious local minima via the
  burer-monteiro approach.
\newblock {\em arXiv preprint arXiv:1609.03240}, 2016.

\bibitem{rudin1973functional}
W.~Rudin.
\newblock Functional analysis, mcgraw-hill series in higher mathematics.
\newblock 1973.

\bibitem{soltanolkotabi2019theoretical}
Mahdi Soltanolkotabi, Adel Javanmard, and Jason~D Lee.
\newblock Theoretical insights into the optimization landscape of
  over-parameterized shallow neural networks.
\newblock {\em IEEE Transactions on Information Theory}, 65(2):742--769, 2019.

\bibitem{srebro2003weighted}
N.~Srebro and T.~Jaakkola.
\newblock Weighted low-rank approximations.
\newblock In {\em Proceedings of the 20th International Conference on Machine
  Learning (ICML-03)}, pages 720--727, 2003.

\bibitem{sun2015matrix}
R.~Sun.
\newblock {\em Matrix completion via nonconvex factorization: Algorithms and
  theory}.
\newblock PhD thesis, University of Minnesota, 2015.

\bibitem{venturi2018spurious}
Luca Venturi, Afonso~S Bandeira, and Joan Bruna.
\newblock Spurious valleys in two-layer neural network optimization landscapes.
\newblock {\em arXiv preprint arXiv:1802.06384}, 2018.

\bibitem{wang2016unified}
L.~Wang, X.~Zhang, and Q.~Gu.
\newblock A unified computational and statistical framework for nonconvex
  low-rank matrix estimation.
\newblock {\em arXiv preprint arXiv:1610.05275}, 2016.

\bibitem{yun2017global}
C.~Yun, S.~Sra, and A.~Jadbabaie.
\newblock Global optimality conditions for deep neural networks.
\newblock {\em arXiv preprint arXiv:1707.02444}, 2017.

\bibitem{yun2018small}
Chulhee Yun, Suvrit Sra, and Ali Jadbabaie.
\newblock Small nonlinearities in activation functions create bad local minima
  in neural networks.
\newblock {\em arXiv preprint arXiv:1802.03487}, 2018.

\bibitem{zhang2019depth}
Li~Zhang.
\newblock Depth creates no more spurious local minima.
\newblock {\em arXiv preprint arXiv:1901.09827}, 2019.

\bibitem{zheng2016convergence}
Q.~Zheng and J.~Lafferty.
\newblock Convergence analysis for rectangular matrix completion using
  burer-monteiro factorization and gradient descent.
\newblock {\em arXiv preprint arXiv:1605.07051}, 2016.

\bibitem{zhu2019distributed}
Zhihui Zhu, Qiuwei Li, Xinshuo Yang, Gongguo Tang, and Michael~B Wakin.
\newblock Distributed low-rank matrix factorization with exact consensus.
\newblock In {\em Advances in Neural Information Processing Systems}, pages
  8422--8432, 2019.

\bibitem{zhu2019global}
Zhihui Zhu, Daniel Soudry, Yonina~C Eldar, and Michael~B Wakin.
\newblock The global optimization geometry of shallow linear neural networks.
\newblock {\em Journal of Mathematical Imaging and Vision}, pages 1--14, 2019.

\bibitem{zou2018stochastic}
Difan Zou, Yuan Cao, Dongruo Zhou, and Quanquan Gu.
\newblock Stochastic gradient descent optimizes over-parameterized deep relu
  networks.
\newblock {\em arXiv preprint arXiv:1811.08888}, 2018.

\end{thebibliography}

\newpage

\begin{appendices}

\section{Proof of Theorem \ref{thm:Loc_Open}}\label{local-open-app}
Before proceeding to the proof of Theorem~\ref{thm:Loc_Open}, we need to state and prove few lemmas. 
\begin{lemma}\label{boundA}
	Let $\bV \in \mathbb{R}^{m \times n}$ be a matrix with $\rank(\bV)=r < m$. Then there exist an index set ${\cal B}=\{1, \ldots , r\} \subseteq \{1,\ldots,m\}$ and a matrix $\bA \in \mathbb{R}^{(m-r) \times r}$  such that 
	\[
	\|\bA\|_{\infty} = \max_{i,j} |\bA_{ij}|\leq 2^{m-r-1}\quad \textrm{and} \quad {\bV}_{{\cal B}^c} = \bA {\bV}_{\cal B},
	\]
	where ${\bV}_{\cal B} \in \mathbb{R}^{r \times n}$ is a matrix with rows $\{{\bV}_{i,:}\}_{i \in {\cal B}}$ and ${\bV}_{{\cal B}^c} \in \mathbb{R}^{(m-r) \times n}$ is a matrix with rows $\{{\bV}_{i,:}\}_{i \in {\cal B}^c}$.
\end{lemma}	
Notice that in the above lemma, the bound on the norm of  matrix $\bA$ is independent of the dimension $n$ and  the choice of matrix $\bV$.
\begin{proof}
	To ease the notation, we denote the $i^{th}$ row of $\bV$ by ${\bv}_i$. We use induction on $m$ to show that there exists a basis ${\cal B}=\{i_1, \ldots, i_r\}$ and a vector $\ba_j \in \mathbb{R}^r$ such that $\forall \, \, j \in {\cal B}^c$, 
	${\bv}_j = \sum_{i \in {\cal B}}{\ba}_{j,i}{\bv}_i$ with  $|{\ba}_{j,i}| \leq 2^{m-r-1} \quad \forall\, \, i \in {\cal B}.$\\
	
	\noindent$\bullet\;$\textit{Induction Base Case } $m=r+1$:
	Without loss of generality, assume ${\cal B}=\{1, \ldots ,r\}$. Since the case of $\bv_{r+1}=\bzero$ trivially holds, we consider $\bv_{r+1} \neq \bzero$. By the property of basis, there exists a non-zero vector $\ba_{r+1} \in \mathbb{R}^r$ such that $\bv_{r+1}=\sum_{i=1}^r {\ba}_{r+1,i}\bv_i$.\\
	Let $i^{*} = \displaystyle{\arg\max_{i \in {\cal B}}} \;|\ba_{r+1,i}|$. If $|{\ba}_{r+1,i^{*}}|  \leq  1,$ then the induction hypothesis is true. Otherwise, when $|{\ba}_{r+1,i^{*}}|  > 1$, we have
	\begin{flalign*}
	{\bv}_{i^{*}}  = \underbrace{\dfrac{1}{{\ba}_{r+1,i^{*}}}}_{\bar{\ba}_{r+1,r+1}} {\bv}_{r+1} - \sum_{i=1;\, i \neq i^{*} }^r \underbrace{\dfrac{{\ba}_{r+1,i}}{{\ba}_{r+1,i^{*}}}}_{\bar{\ba}_{r+1,i}}{\bv}_i 
	= \sum_{i \in {\cal B}^{*}} \bar{\ba}_{r+1,i}{\bv}_i,
	\end{flalign*}
	where ${\cal B}^{*} = \left({\cal B} \cup\{r+1\}\right)\backslash\{i^{*}\},
	$
	i.e., we remove the item $i^*$ from $\mathcal{B}$ and include the item $r+1$ instead. 
	Since $|\bar{\ba}_{r+1,i}| \leq 1 $, the induction base case holds.\\
	
	\noindent$\bullet\;$\textit{Inductive Step:} Assume the induction hypothesis is true for $m>r$, we show it is also true for $m+1$. Without loss of generality we can assume that ${\cal B} = \{1, \ldots ,r\}$. By induction hypothesis,
	${\bv}_j = \sum_{i=1}^r \, {\ba}_{j,i} \,\bv_i$ with $|{\ba}_{j,i}| \leq 2^{m-r-1}, \; \forall \, j=\{r+1, \ldots ,  m\}.$
	Since the case of ${\bv}_{m+1}=\bzero$ trivially holds, we consider ${\bv}_{m+1} \neq \bzero$. Since $\mathcal{B}$ is a basis, there exists ${\ba}_{m+1} \neq \bzero$ such that ${\bv}_{m+1}=\sum_{i=1}^r {\ba}_{m+1,i} \, {\bv}_i$. Let $i^{*} = \displaystyle{\operatornamewithlimits{\mbox{argmax}}_{i \in {\cal B}}}\; |{\ba}_{m+1,i}|$. If $|{\ba}_{m+1,i^{*}}|  \leq 2^{m-r} $, the induction step is done. Otherwise, for the  case of $|{\ba}_{m+1,i^{*}}|  > 2^{m-r},$ we have
	\begin{flalign*}
	{\bv}_{i^{*}}  = \underbrace{\dfrac{1}{{\ba}_{m+1,i^{*}}}}_{\bar{\ba}_{m+1,m+1}} {\bv}_{m+1} - \sum_{i=1;\, i \neq i^{*} }^r \underbrace{\dfrac{{\ba}_{m+1,i}}{{\ba}_{m+1,i^{*}}}}_{\bar{\ba}_{m+1,i}}{\bv}_i 
	= \sum_{i \in {\cal B}^{*}} \bar{\ba}_{m+1,i}{\bv}_i, 
	\end{flalign*}
	where ${\cal B}^{*} = \left({\cal B} \cup \{m+1\}\right)\backslash \{i^{*}\}$ and clearly $|\bar{\ba}_{m+1,i}|\leq 1,$ $\forall \, \, i \in {\cal B}^{*}$ according to the definition of $i^*$. For all $j \in \{r+1 , \ldots , m\}$
	\begin{align}
	{\bv}_{j} & = \sum_{i=1;\, i \neq i^{*} }^r {\ba}_{j,i} \, {\bv}_i 	 + {\ba}_{j,i^{*}}{\bv}_{i^{*}} =  \sum_{ i \neq i^{*} } {\ba}_{j,i} \, {\bv}_i + \dfrac{{\ba}_{j,i^{*}}}{{\ba}_{m+1,i^{*}}} {\bv}_{m+1} - \sum_{i \neq i^{*} } \dfrac{{\ba}_{m+1,i} \, {\ba}_{j,i^{*}}}{{\ba}_{m+1,i^{*}}}{\bv}_i \nonumber\\
	& = \sum_{i=1;\, i \neq i^{*} }^r \big(\underbrace{{\ba}_{j,i} - \dfrac{{\ba}_{j,i^{*}} \, {\ba}_{m+1,i}}{{\ba}_{m+1,i^{*}}}}_{\bar{\ba}_{j,i}}\big) {\bv}_{i} +\underbrace{\dfrac{{\ba}_{j,i^{*}}}{{\ba}_{m+1,i^{*}}}}_{\bar{\ba}_{j,m+1}}{\bv}_{m+1} = \sum_{i \in {\cal B}^{*}} \bar{\ba}_{j,i}{\bv}_i. \nonumber
	\end{align}
	It remains to show that $|\bar{\ba}_{j,i}| \leq 2^{m-r}$ for all $i \in {\cal B^{*}}$, $j \in \{r+1, \ldots , m\}$. Let us first consider $i \in {\cal B^{*}} \backslash \{m+1\}$ and $j \in \{r+1, \ldots , m\}$:
	\begin{flalign*}
	|\bar{\ba}_{j,i} | & \leq |{\ba}_{j,i}| + \Bigl|{\ba}_{j,i^{*}} \dfrac{{\ba}_{m+1,i}}{{\ba}_{m+1,i^{*}}}\Bigr| 
	\leq 2^{m-r-1} + 2^{m-r-1} \Bigl|\dfrac{{\ba}_{m+1,i}}{{\ba}_{m+1,i^{*}}}\Bigr| 
	\leq 2^{m-r} ,
	\end{flalign*}
	where the first inequality holds by triangular inequality, the second inequality holds by the induction hypothesis, and the last inequality holds by the definition of $i^{*}$.
	For $i=m+1$, $|\bar{\ba}_{j,m+1}| = \Bigl|\dfrac{{\ba}_{j,i^{*}}}{{\ba}_{m+1,i^{*}}}\Bigr| \leq \Bigl|\dfrac{2^{m-r-1}}{{\ba}_{m+1,i^{*}}} \Bigr| \leq 2^{m-r}$. This concludes the inductive step and completes our proof. 
	\end{proof}
	
 The following results show that by using local openness of linear mappings and standard {\color{black} SVD}, without loss of generality, we can assume that the product $\bbW_1\bbW_2$ is a diagonal matrix.
 

\begin{lemma}\label{lm:SVD}
	Let ${\bW_1} \in \mathbb{R}^{m \times k}$ and ${\bW_2} \in \mathbb{R}^{k \times n}$. Assume further that  ${\bW_1}{\bW_2} = {\bU} \bSigma {\bV}^{\top}$ is a singular value decomposition of the matrix product $\bW_1\bW_2$ with $\bU \in \mathbb{R}^{m \times m}$, $\bV \in \mathbb{R}^{n \times n}$, and $\bSigma \in \mathbb{R}^{m \times n}$. Then $\mathcal{M}(\cdot,\cdot) \mbox{ is locally open at } ({\bW_1},{\bW_2})$  if and only if  $\mathcal{M}(\cdot,\cdot)$  is locally open at  $({\bU}^{\top}{\bW_1} , {\bW_2}\bV)$.
	
\end{lemma}
The proof of this Lemma is a direct consequence of the definition of local openness. Lemma~\ref{lm:SVD} implies that for  proving Theorem~\ref{thm:Loc_Open}, without loss of generality, we can assume that the product $\bbW_1\bbW_2$ is a diagonal matrix.

 \begin{lemma} \label{lm:TransformationConditions}
 Let ${\bW_1} \in \mathbb{R}^{m \times k}$ and ${\bW_2} \in \mathbb{R}^{k \times n}$. Assume further that  ${\bW_1}{\bW_2} = {\bU} \bSigma {\bV}^{\top}$ is a singular value decomposition of the matrix product $\bW_1\bW_2$ with $\bU \in \mathbb{R}^{m \times m}$, $\bV \in \mathbb{R}^{n \times n}$, and $\bSigma \in \mathbb{R}^{m \times n}$.  Define $\bbW_1 \triangleq \bU^{\top} \bW_1$ and $\bbW_2 \triangleq \bW_2 \bV$. Then the condition $(A)$ below holds true if and only if the condition $(B)$ is true. Similarly, condition $(C)$ is true if and only if condition $(D)$ is true.
 \begin{itemize}
     \item[$(A)$] $\exists \, \widehat{\bW}_1 \in \mathbb{R}^{m\times k} \mbox{ such that } \widehat{\bW}_1\bW_2=\bzero  \mbox{ and } \bW_1 + \widehat{\bW}_1 \mbox{ is full column rank.}$
     \item[$(B)$] $\exists \, \tbW_1 \in \mathbb{R}^{m\times k} \mbox{ such that } \tbW_1\bbW_2=\bzero  \mbox{ and } \bbW_1 + \tbW_1 \mbox{ is full column rank.}$
     \item[$(C)$] $ \exists \,\widehat{\bW}_2 \in \mathbb{R}^{k \times n} \mbox{ such that } \bW_1\widehat{\bW}_2=\bzero  \mbox{ and } \bW_2 + \widehat{\bW}_2 \mbox{ is full row rank.}$
     \item[$(D)$] $ \exists \,\tbW_2 \in \mathbb{R}^{k \times n} \mbox{ such that } \bbW_1\tbW_2=\bzero  \mbox{ and } \bbW_2 + \tbW_2 \mbox{ is full row rank.}$
 \end{itemize}
 \end{lemma}
 \begin{proof}
 Setting $\tbW_1 = \bU^{\top} \widehat{\bW}_1$ and $\tbW_2 = \widehat{\bW}_2 \bV$ leads to the desired result.
 \end{proof}

We next show in Lemma \ref{lm:tri_eq} that if $k < \min\{m,n\}$ and $\rank(\bW_1)=\rank(\bW_2)$, then statements $i,ii,iii,$ and $iv$ in Theorem \ref{thm:Loc_Open} are all equivalent.

\begin{lemma}\label{lm:tri_eq} 
	Let $\bW_1 \in \mathbb{R}^{m \times k}$, $\bW_2 \in \mathbb{R}^{k \times n}$ with $\rank(\bW_1)=\rank(\bW_2)$. Assume further that $k<\min\{m,n\}$.  Then, the following conditions are equivalent 
	\begin{enumerate}[i)]
		\item $\exists \, \tbW_1 \in \mathbb{R}^{m\times k} \mbox{ such that } \tbW_1\bW_2=\bzero  \mbox{ and } \bW_1 + \tbW_1 \mbox{ is full column rank.}$
		\vspace{0.3cm}
		\item $ \exists \,\tbW_2 \in \mathbb{R}^{k \times n} \mbox{ such that } \bW_1\tbW_2=\bzero  \mbox{ and } \bW_2 + \tbW_2 \mbox{ is full row rank.}$
		
		\vspace{0.3cm}
		\item $\dim\big(\, \mathcal{N}(\bW_1) \, \cap \, \mathcal{C}(\bW_2)\big)=0$. 
		\vspace{0.3cm}
		\item$\dim\big( \, \mathcal{N}(\bW_2^{\top}) \, \cap  \, \mathcal{C}(\bW_1^{\top})\, \big)=0$.
	\end{enumerate}
\end{lemma}

\begin{proof}
	To prove the desired result we show the equivalences $ii \Leftrightarrow iii$, and $i \Leftrightarrow iv$. Then we complete the proof by showing $iii \Leftrightarrow iv$. 
	
	We first show the direction $``  ii \Rightarrow iii ''$. Consider ${\bW_1} \in \mathbb{R}^{m \times k}, {\bW_2} \in \mathbb{R}^{k \times n}$  with both being rank~$r$ matrices. Suppose $ii$ holds, then
	\begin{flalign}\label{eq:upperbd1}
	\mathcal{C}({\tbW_2}) \subseteq \mathcal{N}(\bW_1) \mbox{ which implies } \rank(\tbW_2) \leq \dim \big( \, \mathcal{N}(\bW_1) \, \big)=k-r. 
	\end{flalign}
	Also,
	$k=\rank({\bW_2} + \tbW_2) \leq \rank({\bW_2}) + \rank(\tbW_2)=r+\rank(\tbW_2).$
	This inequality combined with (\ref{eq:upperbd1})  implies
	that $\rank(\tbW_2)\, = \, k-r.$
	Note that dim$\big(\, \mathcal{C}(\tbW_2)\, \big)=\text{dim}\big(\, \mathcal{N}(\bW_1) \, \big)$ and $\mathcal{C}(\, \tbW_2 \,) \subseteq \mathcal{N}( \,\bW_1\,)$, which implies that $\mathcal{C}(\, \tbW_2\, )=\mathcal{N}(\, \bW_1 \, ).$ Then, since $\rank(\bW_2 + \tbW_2)= \rank(\bW_2) + \rank(\tbW_2),$ we get
	\[\emptyset =\mathcal{C}(\, \tbW_2\,) \,\cap \, \mathcal{C}(\,\bW_2\,)=\mathcal{N}(\,\bW_1\,)\, \cap \, \mathcal{C}(\, \bW_2\,) \,\Rightarrow \, \text{dim}\big(\, \mathcal{N}(\bW_1) \, \cap \, \mathcal{C}(\, \bW_2 \,)\big)=0.	\]
	We now show the other direction $`` ii \Leftarrow iii ''$. Without loss of generality, let $\bW_2 = \left[ ({\bW_2}^{\prime})^{k \times r}\bA^{r \times n-r} \, \, \, ({\bW_2}^{\prime})^{k \times r} \right]$ where columns of ${\bW_2}^{\prime}$ are linearly independent and let $\tbW_2=\epsilon \left[\bw_{1}^1, \ldots, \bw_{1}^{k-r}, \mathbf{0}, \ldots, \mathbf{0} \right] \in \mathbb{R}^{k \times n}$ be a rank $k-r$ matrix where $\bw_{1}^{i}$ are unit basis of $\mathcal{N}(\, {\bW_1} \,)$ which yields $\mathcal{C}(\, \tbW_2 \,)=\mathcal{N}(\, \bW_1 \, )$. Then since dim$\big({\cal N}(\bW_1) \, \cap \, {\cal C}(\bW_2) \big)=0$, we get rank($\bW_2 + \tbW_2)=k$ for generic choice of $\epsilon$. This completes the proof.\\
	
	Note that by setting ${\bW_1}=\bW_2^{\top}$ and ${\bW_2}=\bW_1^{\top}$, the same proof can be used to show $i \Leftrightarrow iv$. Next, we will prove the equivalence $iii \Leftrightarrow iv$. Notice that
	\[\arraycolsep=1pt\def\arraystretch{1.4}
	\begin{array}{ll}
	\dim\Big(\mbox{span} \big( \mathcal{N}(\bW_1)\cup \mathcal{C}(\bW_2)\big)\Big) & =\dim\left(\mathcal{N}(\bW_1)\right) + \dim\left(\mathcal{C}( \bW_2)\right) - \dim\left(\mathcal{N}( \bW_1 )\cap \mathcal{C}(\bW_2 )\right)  \\
	&= k - r + r - \mbox{dim}\big(\mathcal{N}( \, \bW_1 \,)\cap \mathcal{C}(\,\bW_2 \,)\big)\\
	& =  k - \mbox{dim}\big(\mathcal{N}( \, \bW_1 \,)\cap \mathcal{C}(\,\bW_2 \,)\big).
	\end{array}\]
	Thus,
	\[\arraycolsep=1pt\def\arraystretch{1.4}
	\begin{array}{ll}
	\dim\big(\mathcal{N}( \, \bW_1\,)\cap \mathcal{C}(\,\bW_2\,)\big) \neq 0  &\Leftrightarrow \, \mbox{dim}\Big(\mbox{span}\big(\mathcal{N}( \, \bW_1 \,)\cup \mathcal{C}(\,\bW_2 \,)\big)\Big) < k\\ 
	&\Leftrightarrow  \, \exists \, \ba\neq 0 \mbox{ such that } \ba \, \perp \, \mathcal{C}(\, \bW_2 \,), \mbox{ and }  \ba \, \perp \, \mathcal{N}(\, \bW_1 \,)\\
	&\Leftrightarrow \, \exists \, \ba \neq 0 \mbox{ such that } \ba \in \mathcal{N}(\, \bW_2^{\top} \,), \mbox{ and } \ba \in \mathcal{C}(\, \bW_1^{\top} \,)\\
	&\Leftrightarrow  \mbox{ dim}\big( \mathcal{N}(\, \bW_2^{\top} \big) \, \cap \, \mathcal{C}(\,\bW_1^{\top}\,)\big) \neq 0,   
	\end{array}\]
	which completes the proof.
\end{proof}

\begin{lemma}\label{lm:Y_2=0}
	Let ${\bW_1} \in \mathbb{R}^{m \times k}$, ${\bW_2} \in \mathbb{R}^{k \times n}$ with $k < \min\{m,n\}$ and let $r \triangleq \rank(\bW_1\bW_2)$. Assume further that ${\bW_1}{\bW_2} = \bU\bSigma {\bV}^{\top}$ is an {\color{black}SVD} 
	of $\bW_1\bW_2$ with $\bU \in \mathbb{R}^{m \times m}$, and $\bV \in \mathbb{R}^{n \times n}$, and  $\bSigma \in \mathbb{R}^{m \times n}$. If
	\[\begin{cases}
	i)\, \exists \, \tbW_1 \in \mathbb{R}^{m\times k} \mbox{ such that } \tbW_1{\bW_2}=\bzero  \mbox{ and } {\bW_1} + \tbW_1 \mbox{ is full column rank.}\\
	\hspace{2in} \mbox{and}\\
	ii)\, \exists \, \tbW_2 \in \mathbb{R}^{k\times n} \mbox{ such that } {\bW_1}\tbW_2=\bzero  \mbox{ and } {\bW_2} + \tbW_2 \mbox{ is full row rank.}\\
	\end{cases}\]
	then
	\[\rank(\bW_1) = \rank({\bW_2}), \quad \big({\bW_2}\bV\big)_{:,r+1:n} = \mathbf{0}, \quad  \mbox{and} \quad \big({\bU}^{\top}{\bW_1}\big)_{r+1:n,:} = \bzero.\]
\end{lemma}

\begin{proof}
	Suppose that $ii)$ holds, then
	\begin{flalign}\label{eq:upperbd2}
	\mathcal{C}(\, \tbW_2 \,) \subseteq \mathcal{N}(\, \bW_1 \,) \Rightarrow \rank(\tbW_2) \leq \text{dim}\big(\, \mathcal{N}(\, {\bW_1}\, ) \,\big)=k-\rank({\bW_1}). 
	\end{flalign}
	Also,
	$
	k=\rank({\bW_2}+\tbW_2) \leq \rank({\bW_2}) + \rank(\tbW_2).
	$
	This inequality combined with (\ref{eq:upperbd2}) implies
	\begin{flalign}\label{Ineq1}
	k-\rank({\bW_2}) \leq \rank({\tbW_2}) \leq k-\rank({\bW_1})  \Rightarrow \rank({\bW_2}) \geq \rank({\bW_1}).
	\end{flalign}
	Similarly,  condition $i)$ implies $\rank({\bW_1}) \geq \rank({\bW_2})$. Combined with \eqref{Ineq1}, we obtain $\rank(\bW_1)=\rank({\bW_2})$. Therefore, Lemma~\ref{lm:tri_eq} implies $\textrm{dim}\big(\mathcal{N} ( {\bW_1} )\,  \cap \,  \mathcal{C}(\,{\bW_2})\big)=0$. It  follows from the {\color{black}SVD} 
	of  ${\bW_1}{\bW_2}$, that ${\bU}^{\top}{\bW_1}\big({\bW_2}\bV)_{:,r+1:n}={\bSigma}_{:,r+1:n}=\bzero$, or equivalently ${\bW_1}\big({\bW_2}{\bV}\big)_{:, r+1:n} = \bzero$. On the other hand, since ${\cal C} \big( \, \bW_2{\bV}_{:, r+1:n} \, \big) \subset {\cal C} \big( \, {\bW_2} \, \big)$ and ${\cal N}\big( \, {\bW_1} \, \big) \, \cap \, {\cal C}\big( \, {\bW_2} \, \big) = \emptyset$, we have $\big({\bW_2}\bV)_{:,r+1:n}=\bzero$. Similarly, we can show that $\big({\bU}^{\top}{\bW_1})_{r+1:n,:}=\bzero$.
\end{proof}


\begin{proposition} \label{prop:Non_Sym_case}
	Let  $\mathcal{M} (\bW_1, \bW_2) = \bW_1\bW_2$ be the matrix product mapping with $\bW_1 \in \mathbb{R}^{m \times k}$, $\bW_2 \in \mathbb{R}^{k \times n}$, and $k\, < \, \min\{m, n\}$. Then, $\mathcal{M} (\cdot,\cdot)$ is locally open in its range $\mathcal{R}_{\mathcal{M}}\triangleq \{\bZ\in \mathbb{R}^{m\times n}: \rank(\bZ)\leq k\}$ at the point $(\bbW_1,\bbW_2)$  if and only if the following two conditions are satisfied: 
	\[
	\begin{cases}
	i) \,\exists \, \tbW_1 \in \mathbb{R}^{m \times k} \mbox{ such that } \tbW_1 \bbW_2 =\bzero  \mbox{ and } \bbW_1 + \tbW_1  \mbox{ is full column rank.}\\
	\hspace{2in} \mbox{and}\\
	ii)\, \exists \, \tbW_2 \in \mathbb{R}^{k\times n} \mbox{ such that } \bbW_1\tbW_2=\bzero  \mbox{ and } \bbW_2 + \tbW_2 \mbox{ is full row rank.}
	\end{cases}
	\]
\end{proposition}

\begin{proof}
	First of all, according to Lemma~\ref{lm:SVD} and Lemma~\ref{lm:TransformationConditions}, without loss of generality we can assume that the matrix product $\bbW_1 \bbW_2$ is of diagonal form.
	
	Let us start by first proving the ``only if'' direction. Notice that the result clearly holds when $\rank(\bbW_1)=\rank(\bbW_2)=k$ by choosing $\tbW_1=\tbW_2 = \bzero$.  Moreover, the mapping $\mathcal{M}(\cdot , \cdot)$ cannot be locally open if only one of the matrices $\bbW_1$ or $\bbW_2$ is rank deficient. To see this, let us assume that $\bbW_1$ is full column rank, while $\bbW_2$ is rank deficient. Assume further that the mapping ${\cal M}(\cdot , \cdot)$ is locally open at $(\bbW_1,\bbW_2)$, it follows from the definition of openness that the mapping ${\cal M}^1(\bW_1,\bW_2^1) \triangleq \bW_1\bW_2^1$ is locally open at $(\bbW_1, \bbW_2^1)$ where $\bbW_2^1 \triangleq (\bbW_2)_{:,1:k}$ only contains  the first $k$ columns of $\bbW_2$.
	Since the range of the mapping ${\cal M}^1$ at $(\bbW_1,\bbW_2^1)$ is the entire space $\mathbb{R}^{m \times k}$,  Proposition \ref{Prev_result} implies that
	\[
	\begin{cases}
	\exists \, \tbW_1  \mbox{ such that } \tbW_1\bbW_2^1=\bzero  \mbox{ and } \bbW_1 + \tbW_1 \mbox{ is full row rank.}\\
	\hspace{1.8in}  \mbox{or}\\
	\exists \, \tbW_2^1 \mbox{ such that } \bbW_1\tbW_2^1=\bzero  \mbox{ and } \bbW_2^1 + \tbW_2^1 \mbox{ is full rank.}\\
	\end{cases}
	\]
	
	Moreover, since $\bbW_1 \in \mathbb{R}^{m \times k}$ and $m>k$, it is impossible for $\tbW_1 + \bbW_1$ to be full row rank. On the other hand, since $\bbW_1$ is full column rank, $\bbW_1\tbW_2^1=\mathbf{0}$ implies that $\tbW_2^1=\mathbf{0}$; and hence $\bbW_2^1 + \tbW_2^1$ is not full column rank. Hence none of the above two conditions can hold and consequently, $\mathcal{M}(\cdot , \cdot)$ cannot be open at the point $(\bbW_1,\bW_2^1)$ in this case. Similarly, we can show that when $\bbW_1$ is rank deficient and $\bbW_2$ is full row rank, the mapping $\mathcal{M}(\cdot, \cdot)$ cannot be locally open. Hence, if $\bbW_1$ and $\bbW_2$ are not both full rank, then they both should be rank deficient.
	
	Assume that the matrices $\bbW_1$ and $\bbW_2$ are both rank deficient and $\mathcal{M}(\cdot , \cdot)$ is locally open at $(\bbW_1,\bbW_2)$. It  follows that ${\cal M}^1(\bW_1,\bW_2^1) \triangleq \bW_1\bW_2^1$ is locally open at $(\bbW_1,\bbW_2^1)$. By Proposition \ref{Prev_result}, and since there does not exist $\tbW_1$ such that $\bbW_1+\tbW_1$ is full row rank, there should exist $\tbW_2^1$ such that $\bbW_1 \tbW_2^1=\bzero$ and $\bbW_2^1+\tbW_2^1$ is full rank. Defining $\tbW_2\triangleq \left[
	\begin{array}{c|c}
	\tbW_2^1\, &\, \bzero
	\end{array}
	\right]$, we satisfy the desired condition $ii)$. Similarly, by looking at the transpose of the mapping $\mathcal{M}$, we can show that condition $i)$ is true when $\mathcal{M}$ is locally open.

	We now prove the ``if'' direction. Suppose $i)$ and $ii)$ hold.
	
	Let $\bSigma=\bbW_1\bbW_2 = \left[\begin{array}{cc} {\bSigma}_{: , 1:r} & \bzero \end{array}\right]$ be a rank $r$ matrix. Lemma~\ref{lm:Y_2=0} implies that $\rank(\bbW_1)= \rank(\bbW_2)$, and the last $n-r$ columns of $\bbW_2$ are all zero. We need to show that for any given $\epsilon > 0$, there exists $\delta > 0$, such that 
	${\ball}_{\delta}\big(\bbW_1\bbW_2 \big) \, \cap {\cal R}_{\cal M} \subseteq {\cal M}\Big({\ball}_{\epsilon}\big(\bbW_1\big),{\ball}_{\epsilon}\big(\bbW_2\big)\Big).$ Consider a perturbed matrix $\widetilde{\bSigma} \in {\ball}_{\delta}\big(\boldsymbol{\Sigma} \big) \, \cap \, {\cal R}_{\cal M}$, we show that $\widetilde{\bSigma} \in \mathcal{M}\Big({\ball}_{\epsilon}\big(\bbW_1\big),{\ball}_{\epsilon}\big(\bbW_2\big)\Big)$. Without loss of generality, and by permuting the columns of $\widetilde{\bSigma}$ if necessary, any perturbation of $\bsigma$ with rank at most $k$ can expressed as
	\[\widetilde{\bSigma} \, = \,
	\left[
	\begin{array}{c|c|c}
	\, \underbrace{{\bSigma}_{:,1:r} + {\bR}_{\delta}^{1}}_{m \times r} & 
	\underbrace{{\bR}_{\delta}^{2}}_{m \times (k-r)} &
	\underbrace{({\bSigma}_{:,1:r} + {\bR}_{\delta}^{1}){\bA}_1 + {\bR}_{\delta}^{2}{\bA}_2}_{m \times (n-k)} 
	\end{array}
	\right].\]
	Here ${\bA}_1 \in \mathbb{R}^{r \times (n-k)}$ and ${\bA}_2 \in \mathbb{R}^{(k-r) \times (n-k)}$ exist since rank($\widetilde{\bSigma}$)$\leq k$.  {\color{black}More specifically, these matrices exist since each of the last $n-k$ columns of $\widetilde{\bSigma}$ is a linear combination of the first $k$ columns. Moreover, let the perturbation matrix $\bR_{\delta}$ be defined as  $\bR_{\delta}\triangleq \|\widetilde{\bSigma} - \bSigma\|$. Then $\bR_{\delta}$ can expressed as}
	\[
	{\bR}_{\delta} \triangleq \left[
	\begin{array}{c|c|c}
	{\bR}_{\delta}^{1}\, &\, {\bR}_{\delta}^{2} & ({\bSigma}_{:,1:r} + {\bR}_{\delta}^{1}){\bA}_1 + {\bR}_{\delta}^{2}A_2
	\end{array}
	\right]
	\]
	and requires to have a norm less than or equal $\delta$, i.e., $\|{\bR}_{\delta}\|\leq \delta$. Since $\rank(\bbW_2 + \tbW_2) = k$, there exist a unitary basis set $\{\tbw_2^1,\ldots,\tbw_2^{k-r}\}$ for $\tbW_2$ such that $\mbox{span}\{\tbw_2^1,\ldots,\tbw_2^{k-r}\} \, \cap \, {\cal C}(\bbW_2) = \emptyset$. Define
	\begin{equation}\label{eq:PertW21}
	\tbW_2^1 \triangleq \dfrac{\epsilon}{n2^{n+1}} \left[ \overbrace{\mathbf{0}}^{k \times r} \, \, \overbrace{\tbw_2^1 \, \ldots , \tbw_2^{k-r}}^{k \times (k-r)} \right],
	\end{equation}
	and let us form the matrix $\bbW_2^1\in \mathbb{R}^{k \times k}$ using the first $k$ columns of $\bbW_2$. Since the last $n-r$ columns of the matrix $\bbW_2$ are zero, $\tbW_2^1 + \bbW_2^1$ is a full rank $k \times k$ matrix and $\bbW_1\tbW_2^1=0$. 
	Let us define
	$
	\bbW_1^0\triangleq \left[\begin{array}{c|c} {\bR}_{\delta}^{1} & {\bR}_{\delta}^{2} \end{array} \right](\tbW_2^1 + \bbW_2^1)^{-1},
	$
	and
	\small
	$
	\bbW_2^0 \triangleq \left[
	\begin{array}{c|c}
	\tbW_2^1 & \big(\bbW_2^1 + \tbW_2^1 \big)_{:,1:r}{\bA}_1 + \big(\bbW_2^1 + \tbW_2^1\big)_{:, r+1:k}{\bA}_2
	\end{array}
	\right].
	$
	\normalsize
	Using this definition, we have
	\small
	\begin{flalign}\label{perm_equ}
	(\bbW_1+\bbW_1^0)(\bbW_2 + \bbW_2^0)&= \left[(\bbW_1+\bbW_1^0)(\bbW_2 + \bbW_2^0)_{:,1:k} \,\, \Big|  \, \, (\bbW_1+\bbW_1^0)(\bbW_2 + \bbW_2^0)_{:,k+1:n}\right]\nonumber \\
	&= \left[(\bbW_1+\bbW_1^0)(\bbW_2^1 + \tbW_2^1) \,\, \Big|  \, \, (\bbW_1+\bbW_1^0)(\bbW_2 + \bbW_2^0)_{:,k+1:n} \right]\nonumber \\
	&= \left[\begin{array}{c|c}\bar{\bSigma}_{:,1:k} + \underbrace{\bbW_1\tbW_2^1}_{=0} + \left[\begin{array}{c|c} {\bR}_{\delta}^{1} & {\bR}_{\delta}^{2} \end{array}\right](\bbW_2^1 + \tbW_2^1)^{-1}(\bbW_2^1 + \tbW_2^1) & \underbrace{\bzero}_{m \times (n-k)}\end{array}  \right]\nonumber \\
	&+ \left[ \begin{array}{c|c}\underbrace{\bzero}_{m \times k} &  
	(\bbW_1+\bbW_1^0)\left[ \left[ 
	\begin{array}{cc}\big(\bbW_2^1 + \tbW_2^1 \big)_{:,1:r} \hspace{0in} & \hspace{0in}  \big(\bbW_2^1 + \tbW_2^1\big)_{:, r+1:k}\end{array}
	\right] \left[
	\begin{array}{c} \bA_1 \\ \bA_2 \end{array} 
	\right] \right] \end{array} \right]\nonumber \\
	&=\bbW_1\bbW_2 + \left[
	\begin{array}{c|c|c}
	{\bR}_{\delta}^{1} & {\bR}_{\delta}^{2}  & ({\bSigma}_{:,1:r} + {\bR}_{\delta}^{1}){\bA}_1 + {\bR}_{\delta}^{2}{\bA}_2\end{array}\right]\nonumber \\
	&=\bbW_1\bbW_2 + {\bR}_{\delta} = \widetilde{\bSigma}. 
	\end{flalign}
	\normalsize
	To complete the proof, it remains to show that for any $\epsilon > 0$, we can choose $\delta$ small enough  such that $\| \bbW_1^0 \| \leq \epsilon$ and $\| \bbW_2^0 \| \leq \epsilon$. In other words, we will show  $\widetilde{\bSigma} \in {\cal M}\Big({\ball}_{\epsilon}\big(\bbW_1\big),{\ball}_{\epsilon}\big(\bbW_2\big)\Big)$. 
	Let $\widetilde{r}$, with $k\geq\widetilde{r}\geq r$, be the rank of $\widetilde{\bSigma}$.  According to Lemma~\ref{boundA} and by possibly permuting the columns, $\widetilde{\bSigma}$ can be expressed as
	$\widetilde{\bSigma} = \left[ \begin{array}{c|c} \widetilde{\bSigma}_1 \,&\, \widetilde{\bSigma}_1\bbA
	\end{array}
	\right],$
	where $\widetilde{\bSigma}_1 \in \mathbb{R}^{m \times \widetilde{r}}$ is full column rank, and $\bbA$ has a bounded norm $\|\bbA \| \leq n2^{n -\widetilde{r}-1}$. Notice that for given $\bbW_1^0$ and $\bbW_2^0$  satisfying (\ref{perm_equ}), permuting the columns of $\widetilde{\bSigma}$ corresponds to permuting the columns of $(\bbW_2+ \bbW_2^0)$.  If we can show that the first $r$ columns are not among the permuted ones, then using the fact that $\bbW_2$ has only its first $r$ columns non-zero, it follows that the permutation of the columns of $\widetilde{\bSigma}$ corresponds to the same permutation of the columns of $\bbW_2^0$. 
	Moreover, if the first $r$ columns are not among the permuted ones, then without loss of generality we can express the perturbed matrix
	\[\widetilde{\bSigma} \, = \,
	\left[
	\begin{array}{c|c|c}
	\, \underbrace{{\bSigma}_{:,1:r} + {\bR}_{\delta}^{1}}_{m \times r} & 
	\underbrace{{\bR}_{\delta}^{2}}_{m \times (k-r)} &
	\underbrace{({{{\Sigma}}_{:,1:r} + {\bR}_{\delta}^{1})\bbA_1 + {\bR}_{\delta}^{2}\bbA_2}}_{m \times (n-k)} \\
	\end{array}
	\right],
	\]
	and the perturbation matrix
	\[\bR_{\delta} \, = \,
	\left[
	\begin{array}{c|c|c}
	\, \underbrace{{\bR}_{\delta}^{1}}_{m \times r} \, & 
	\underbrace{{\bR}_{\delta}^{2}}_{m \times (k-r)} &
	\underbrace{({{{\Sigma}}_{:,1:r} + {\bR}_{\delta}^{1})\bbA_1 + {\bR}_{\delta}^{2}\bbA_2}}_{m \times (n-k)} \\
	\end{array}
	\right],
	\]
	where $\left[\begin{array}{c}
	\bbA_1\\
	\bbA_2
	\end{array}
	\right]=\bbA$ has a bounded norm.
	
	We now show that the first $r$ columns of $\widetilde{\bSigma}$ 
	are not among the permuted columns. Assume the contrary, then there exists at least a column ${\bSigma}_{:,j} + \big({\bR}_{\delta}^{1}\big)_{:,j}$ with $j\leq r$, that is not a column of $\widetilde{\bSigma}_1$ and is thus a column of $\widetilde{{\bSigma}}_1\bbA$. Without loss of generality let ${\bSigma}_{:,j} + \big({\bR}_{\delta}^{1}\big)_{:,j} = \widetilde{\bSigma}_1\bbA_{:,1}$. It follows that 
	${\bSigma}_{j,j} + \big({\bR}_{\delta}^{1}\big)_{j,j} = (\widetilde{\bSigma}_1)_{j,:}\bbA_{:,1}.$
	But since ${\bSigma}_{j,j} + \big({\bR}_{\delta}^{1}\big)_{j,j}$ is a non-zero perturbed singular value, and since elements of $(\widetilde{\bSigma}_1)_{j,:}$ are all of order $\delta$, then by choosing $\delta$ sufficiently small, we get $\|\bbA\| > 2^{n - \widetilde{r}-1}$, which contradicts the bound we have on $\bbA$. 
	
	We now obtain an upper-bound on $\|\bbW_2^0\|$. Since the norm of $\bbA$ is bounded,  the norm of $\bbA_2$ is also bounded by some constant~$K \triangleq  n2^n > n2^{n-\widetilde{r}-1}$. Hence,
	\begin{flalign*}
	\delta & \geq \|\bR_{\delta}\|\geq \|({\bSigma}_{:,1:r} + {\bR}_{\delta}^{1})\bbA_1 + {\bR}_{\delta}^{2}\bbA_2\| \geq \|({\bSigma}_{:,1:r} + {\bR}_{\delta}^{1})\bbA_1\|  - \|{\bR}_{\delta}^{2}\bbA_2\| \\
	& \geq \|({\bSigma}_{:,1:r} + {\bR}_{\delta}^{1})\bbA_1\|  - K\delta \geq \frac{{\sigma}_{\min}}{2}\|\bbA_1\|  - K\delta,
	\end{flalign*}
	where ${\sigma}_{\min}$ is the minimum singular value of the full column rank matrix ${\bSigma}_{:,1:r}$ which is bounded away from zero. Here, we have chosen $\delta<\sigma_{\min}/2$ so that $\|({\bSigma}_{:,1:r} + {\bR}_{\delta}^{1})\bbA_1\| \geq  \dfrac{{\sigma}_{\min}}{2}\|\bbA_1\| $. Rearranging  the terms, we obtain
	$
	\|\bbA_1\| \leq \dfrac{2 (1+K)\delta}{{\sigma}_{\min}}.
	$
	Thus, for some constant $C \triangleq \|\bbW_2^1\|$, we obtain
	\begin{flalign*}
	\|\bbW_2^0\|^2 &\leq \|\tbW_2^1\|^2 + \|\bbW_2^1\|^2 \, \|\bbA_1\|^2 + \|\tbW_2^1\|^2 \,\|\bbA_2 \|^2  \\
	& \leq \dfrac{\epsilon^2}{4n^22^{2n}} + {\delta}^2C^2 \Bigg(\dfrac{2+2K}{{\sigma}_{min}}\Bigg)^2 +\dfrac{\epsilon^2 K^2}{4n^22^{2n}} \\
	& \leq \dfrac{\epsilon^2 }{4K^2} +{\delta}^2 C^2\Bigg(\dfrac{2+2K}{{\sigma}_{min}}\Bigg)^2 + \dfrac{\epsilon^2 K^2}{4K^2}\\  
	&\leq \epsilon^2/2  +{\delta}^2C^2 \Bigg(\dfrac{2+2K}{{\sigma}_{min}}\Bigg)^2,
	\end{flalign*}
	where the first inequality holds by Chauchy Swarz and triangular inequality. Thus, for a given $\epsilon>0$, we can choose 
	\[
	\delta \leq \min\left\{\dfrac{\epsilon}{1+\max\left\{\|(\bbW_2^1+\tbW_2^1)^{-1}\|,\sqrt{2} \, C \Bigg(\dfrac{2+2K}{{\sigma}_{min}}\Bigg)\right\}}, \sigma_{\min}/2\right\}.
	\]
	This choice of $\delta$ leads to  $\|\bbW_2^0\| \leq \epsilon$. Moreover,
	\begin{align}
	\|\bbW_1^0\| &\leq \|{\bR}_{\delta}\|\,\|(\bW_2^1 + \tbW_2^1)^{-1}\| \leq {\delta} \|(\bW_2^1 + \tbW_2^1)^{-1}\|
	\leq \dfrac{\epsilon \|(\bW_2^1 + \tbW_2^1)^{-1}\|}{1+\|(\bW_2^1 + \tbW_2^1)^{-1}\|} 
	\leq \epsilon, \nonumber
	\end{align}
	which completes the proof.
	\end{proof}

We now use Proposition \ref{prop:Non_Sym_case}, Lemma \ref{lm:tri_eq}, and  Lemma \ref{lm:Y_2=0} to complete the proof of Theorem~\ref{thm:Loc_Open}:

\begin{proof}
	First of all, if $\mathcal{M}(\cdot, \cdot)$ is locally open at $(\bbW_1,\bbW_2)$,  according to Proposition~\ref{prop:Non_Sym_case}, the conditions $i)$ and $ii)$ must hold; and hence $\rank(\bbW_1) = \rank(\bbW_2)$ due to Lemma~\ref{lm:Y_2=0}. Thus, $\mathcal{M}(\cdot,\cdot)$ cannot be locally open if $\rank(\bbW_1) \neq \rank(\bbW_2)$.  On the other hand, when $\rank(\bbW_1)= \rank(\bbW_2)$, the conditions $i)$, $ii)$, $iii)$, and $iv)$ are equivalent due to Lemma~\ref{lm:tri_eq}. Moreover, these conditions imply local openness according to Proposition~\ref{prop:Non_Sym_case}. 
	\end{proof}

\section{Proof of Theorem \ref{thm:Loc_Open_general}}\label{PSD-local-open-app}
Before proceeding with the proof of Theorem~\ref{thm:Loc_Open_general}, we recall the definition of the symmetric matrix multiplication mapping
$
\mathcal{M}_{+}: \mathbb{R}^{n \times k} \mapsto \mathcal{R}_{\mathcal{M}_+}  \quad \mbox{with} \quad \mathcal{M}_{+}(\bW) \triangleq \bW\bW^{\top}.$
where $ \mathcal{R}_{\mathcal{M}_+} \triangleq  \{\, \bZ \in \mathbb{R}^{n \times n} \,\, | \, \, \bZ\succeq 0, \, \, \mbox{rank}(\bZ) \, \leq \, k \}$. In this section, we show that $\mathcal{M}_{+}$ is open in $\mathcal{R}_{{\cal M}_+}$. Particularly, we show that given a matrix $\bW \in \mathbb{R}^{ n\times k}$ and a small perturbation $\widetilde{\bZ} \, \in \, \mathcal{R}_{{\cal M}_{+}}$ of $\bZ \triangleq \bW\bW^{\top}$, there exists a small perturbation $\tbW$ of $\bW$ such that $\widetilde{\bZ}\,=\,\tbW\tbW^{\top}$. Similar to the previous proof scheme, we first show that local openness of ${\cal M}_{+}(\cdot)$ at $\bW$ is equivalent to local openness of ${\cal M}_{+}(\cdot)$ at $\bU^{\top}\bW$ where $\bU^{\top}\bSigma\bU$ is a symmetric singular value decomposition of the product $\bW\bW^{\top}$.

\begin{lemma}\label{Sym_lm:SVD}
	Consider $\bW \in \mathbb{R}^{n\times k}$ and assume that $\bW \bW^{\top} =\bU\bSigma\bU^{\top}$ is a symmetric singular value decomposition of the matrix product $\bW\bW^{\top}$  with $\bU \in \mathbb{R}^{n \times n}$, and $\bSigma \in \mathbb{R}^{n \times n}$ . Then,
	$\mathcal{M}_{+}(\cdot)$ is locally open at $ \bW$  if and only if $  \mathcal{M}_{+}(\cdot)$  is locally open at $ \bU^{\top}\bW.$
\end{lemma}
The proof of this lemma is a direct consequence of local openness definition.

According to Lemma \ref{Sym_lm:SVD}, proving local openness of $\mathcal{M}_{+}(\cdot)$ at $\bW$ is equivalent to proving local openness of $\mathcal{M}_{+}(\cdot)$ at $\bU^{\top}\bW$. To ease the notation, denote $\bU^{\top}\bW$ by $\bbW$. Notice that when $\bbW \in \mathbb{R}^{n \times n}$ is a full rank square matrix,  for any symmetric perturbation  $\bR_{\delta}$ with $\|\bR_{\delta}\| \leq \delta$ sufficiently small, $\widetilde{\bSigma}=\bbW\bbW^{\top} + \bR_{\delta}$ is a full rank symmetric positive definite matrix. Then finding a perturbation $\bbW + \bA_{\epsilon}$ of $\bbW$ such that $(\bbW + \bA_{\epsilon})(\bbW + \bA_{\epsilon})^{\top} = \widetilde{\bSigma}$ is equivalent to solving the  matrix equation 
$\bA_{\epsilon}\bbW^{\top} + \bbW\bA_{\epsilon}^{\top} + \bA_{\epsilon}\bA_{\epsilon}^{\top} = \bR_{\delta}.$ Substituting $\bA_{\epsilon}=\bP(\bbW^{-1})^{\top}$ for some matrix $\bP \in \mathbb{R}^{n \times n}$, we obtain the following quadratic matrix  equation
\begin{flalign}\label{P_eqn}
\bP + \bP^{\top} + \bP{\bSigma}^{-1}\bP^{\top} = \bR_{\delta},
\end{flalign}
where $\bSigma = \bbW \bbW^{\top}$.
In the next Lemma, we show how to find a solution matrix $\bP$ with $\|\bP\|={\cal O}(\delta)$ that satisfies (\ref{P_eqn}); thus proving local openness of $\mathcal{M}_{+}(\cdot)$ at any full rank square matrix $\bbW$.

\begin{lemma}\label{P_matrix}
	Let $\bSigma\in \mathbb{R}^{n\times n}$ be a full rank diagonal positive definite matrix. There exists $\delta_0>0$ such that for any positive $\delta<\delta_0$ and any symmetric matrix $\bR \in \mathbb{R}^{n \times n}$ with $\|\bR\|_{\infty} \leq \delta$, there exists an upper-triangular matrix $\bP \in \mathbb{R}^{n \times n}$ with $\|\bP\|_{\infty} \leq 3\delta$ satisfying the  equation
	$
	\bP+\bP^{\top} + \bP\bSigma^{-1}\bP^{\top}=\bR.
	$
\end{lemma}
Before proving this lemma, let us emphasize that the value of $\delta_0$ depends on $\bSigma$, but is independent of the choice of $\bR$. 
\begin{proof}
	Let us start by simplifying the equation of interest. For all $i =1 , \ldots,n$, let $s_i={\bSigma}_{ii}^{-1}$, which is positive by the positive definiteness of ${\bSigma}$. Then,
	\[\arraycolsep=1pt\def\arraystretch{1.4}
	\begin{array}{ll}
	\bP+\bP^{\top}+\bP{\bSigma}^{-1}\bP^{\top} = \bR &\Leftrightarrow
	\begin{cases} 
	2\bP_{ii} \, + \,\displaystyle{\sum_{l}}\,s_l\bP_{il}^2 \,=\, \bR_{ii}\, \quad \forall \, i\\ 
	\bP_{ij} + \bP_{ji} + \displaystyle{\sum_{l}}s_l\bP_{il}\bP_{jl}=\bR_{ij}\, \quad \forall \, i<j\\
	\end{cases}\\
	&\Leftrightarrow
	\begin{cases} 
	\big(s_i\bP_{ii}\, + \, 1\big)^2 \, + \, \displaystyle{\sum_{l \neq i}}\,s_is_l\bP_{il}^2 \,=\, s_i\bR_{ii}+1 \, \quad \forall \, i\\ 
	\bP_{ij}\big(\, s_j\bP_{jj}\,+\,1\big) + \bP_{ji}\big(s_i\bP_{ii} \, + \, 1 \big) + \displaystyle{\sum_{l \neq i,j}}\,s_l\bP_{il}\bP_{jl}=\bR_{ij}\, \quad \forall \, i<j
	\end{cases}\\
	&\Leftrightarrow
	\begin{cases} 
	\bP_{ii} = \dfrac{1}{s_i} \left( \pm\displaystyle{\sqrt{s_i\bR_{ii} + 1 -\displaystyle{\sum_{l \neq i}}\,s_is_l\bP_{il}^2}} -1 \right)\\
	\bP_{ij}\big(\, s_j\bP_{jj}\,+\,1\big) + \bP_{ji}\big(s_i\bP_{ii} \, + \, 1 \big) + \displaystyle{\sum_{l \neq i,j}}\,s_l\bP_{il}\bP_{jl}=\bR_{ij}\, \quad \forall \, i<j.
	\end{cases}
	\end{array}\]
	
	An upper-triangular solution $\bP$ can be generated using the following pseudo-code:
	
	\begin{algorithm}[H]\label{alg-P}
		\caption{Pseudo-code for generating matrix $\bP$}\label{P_Gen_Code}
		\begin{algorithmic}[1]
			\State For all $(i,j)$ with $i>j$, set $\bP_{ij}=0$.
			\For{$j=n \rightarrow 1$}
			\small \begin{equation}\label{alg_eqn_1}
			\bP_{jj} = \dfrac{1}{s_j} \Bigg| \displaystyle{\sqrt{s_j\bR_{jj} + 1 -\displaystyle{\sum_{l > j}}\,s_js_l\bP_{jl}^2}} -1 \Bigg|
			\end{equation}\normalsize
			\For{$i=j-1 \rightarrow 1$}
			\small \begin{equation}\label{alg_eqn_2}
			\bP_{ij}=\dfrac{\bR_{ij} - \sum_{l> j}s_{l}\bP_{il}\bP_{jl}}{s_{j}\bP_{jj} +1}
			\end{equation} \normalsize
			\EndFor
			\EndFor
		\end{algorithmic}
	\end{algorithm}
	Notice that at each iteration of the algorithm corresponding to the $(i,j)$-th index, the corresponding equation is satisfied. Moreover, once an equation is satisfied, the variables in that equation are not going to change anymore; and thus it remains satisfied.  
	We proceed by showing that Algorithm (\ref{P_Gen_Code}) generates a matrix $\bP$ with $\|\bP\|\leq 3\delta $ for $\delta$ small enough. In particular, we show that for sufficiently small $\delta>0$, $|\bP_{ij}| \leq 2 \delta + {\cal O}({\delta}^2)$ for all $i\leq j$. We prove our result by a reverse induction on $j$:\\
	
	\noindent\underline{Base step, $j=n$ (last column of $\bP$):} Using (\ref{alg_eqn_1}),
	\[|\bP_{nn}|=\dfrac{1}{s_{n}}\big|\sqrt{s_{n}\bR_{nn} +1} -1 \big|\leq \dfrac{1}{s_{n}}\left(s_{n}|\bR_{nn}| +1 -1 \right)=|\bR_{nn}| \leq \delta.\]
	Moreover, \eqref{alg_eqn_2}) implies
	$
	|\bP_{in}|=\dfrac{|\bR_{in}|}{|s_{n}\bP_{nn} +1|}.
	$
	For sufficiently small $\delta, \, \, |s_{n}\bP_{nn} + 1| \geq \dfrac{1}{2}.$ It follows that  $ |\bP_{in}| \leq  2|\bR_{in}| \leq 2\delta.$\\
	
	\noindent\underline{Induction hypothesis:} Assume $|\bP_{ij}| \leq 2 \delta + {\cal O}({\delta}^2)$ for all $i\leq j$, $j = n, \ldots, k $. We show that the result holds for $k-1$. First of all,   (\ref{alg_eqn_1}) implies 
	\small
	\begin{flalign*}
	& |\bP_{(k-1)(k-1)}| = \dfrac{1}{s_{k-1}}\Bigg|\sqrt{s_{k-1}\bR_{(k-1)(k-1)} +1 - \sum_{l > k-1} s_{k-1}s_{l}\bP_{(k-1)l}^2} -1\Bigg| \\
	\leq &  \dfrac{1}{s_{k-1}}\Bigg| s_{k-1}|\bR_{(k-1)(k-1)}| +1 + \big|\sum_{l > k-1} s_{k-1}s_{l}\bP_{(k-1)l}^2 \big| -1\Bigg| \leq |\bR_{(k-1)(k-1)}| + {\cal O}({\delta}^2).
	\end{flalign*}
	\normalsize
	Also,
	$
	|\bP_{i(k-1)}|=\dfrac{\bigg|\bR_{i(k-1)} - \sum_{l>k-1}s_{l}\bP_{il}\bP_{(k-1)l}\bigg|}{\bigg|s_{k-1}\bP_{(k-1)(k-1)} +1\bigg|}, 
	$
	which implies
	\[ 
	|\bP_{i(k-1)}|\, \leq \, \dfrac{|\bR_{i(k-1)}| + |\sum_{l>k-1}s_{l}4{\delta}^2 + {\cal O}({\delta}^3)|}{|s_{k-1}\bP_{(k-1)(k-1)} +1|}.  
	\]
	Thus, for sufficiently small $\delta$, we have $|s_{k-1}\bP_{(k-1)(k-1)} + 1 |\geq \dfrac{1}{2}.$ Consequently, $|\bP_{i(k-1)}| \leq 2|\bR_{i(k-1)}| + {\cal O}({\delta}^2)  \leq 2\delta + {\cal O}({\delta}^2).$ 
	\end{proof}

Using the above two lemmas we complete the proof of Theorem~\ref{thm:Loc_Open_general}, i.e showing local openness of ${\cal M}_{+}$  at non-square matrices $\bbW$.

\begin{proof}\bf{Proof for the Local Openness of ${\cal M}_{+}$  at non-square matrices $\bbW$.\\}

	To show the openness of the mapping, it suffices to show that it is locally open everywhere. Consider an arbitrary point $\bW \in \mathbb{R}^{n \times k}$, and let $\bU\bSigma\bU^{\top}$ be a singular value decomposition of the symmetric matrix product $\bW\bW^{\top}$. To ease the notation, denote $\bU^{\top}\bW$ by $\bbW$. By Lemma \ref{Sym_lm:SVD}, $\mathcal{M}_{+}(\cdot)$ is locally open at $\bW$ if and only if $\mathcal{M}_{+}(\cdot)$ is locally open at $\bbW$.
	When $\bW\bW^{\top}$ is rank deficient, we can write
	$
	\bSigma=\bbW\bbW^{\top} = 
	\left[
	\begin{array}{c  c} 
	\, \bSigma_1 & \bzero\\ 
	\bzero & \bzero 
	\end{array} 
	\right],$
	where $\bSigma_1 \in \mathbb{R}^{r \times r}$ is a positive definite diagonal matrix and $r$ is the rank of $\bW\bW^{\top}$. It is easy to show that the last $n-r$ rows of $\bbW$ are all zeros, i.e., for all $j > r$, 
	$
	\langle \bbW_{j,:},\big(\bbW_{:,j}\big)^{\top} \rangle = \|\bbW_{j,:}\|^2 =0,$ or equivalently, $ \bbW_{j,:}=\mathbf{0}.$
	To show local openness of ${\cal M}_{+}(\cdot)$ at $\bbW$, we consider a perturbation $\widetilde{\bSigma} \triangleq \bSigma + \bR_{\delta}$ of $\bSigma$ in the range ${\cal R}_{{\cal M}_{+}}$, and show that there exists a small perturbation $\bbW + \bA_{\epsilon}$ of $\bbW$ such that $(\bbW + \bA_{\epsilon})(\bbW + \bA_{\epsilon})^{\top} = \widetilde{\bSigma}$. By possibly permuting the columns of $\widetilde{\bSigma}$, the perturbed matrix which we know is symmetric positive semi-definite with rank at most $k$ can be expressed as
	
	\footnotesize
	\[
	\begin{array}{c c}
	& \begin{array}{c c c}  \hspace{-0.3in} r \mbox{ columns}  \hspace{-0.3in}   &   \hspace{0.5in}   k-r\mbox{ columns}   \hspace{0.5in}  &   \hspace{0.5in}   n-k \mbox{ columns} \end{array} \\
	\\
	\begin{array}{c} r\mbox{ rows} \\ \vspace{0.1in}\\ k-r\mbox{ rows}  \\ \vspace{0.25in}\\n-k\mbox{ rows}  \\ \end{array}
	& \left[\begin{array}{c  c  c}
	{{\bSigma}}_1 + \bR_{1} \hspace{0in} &  \hspace{0.in} \bR_2 \hspace{0.in} &  \hspace{0.in}  \left[\begin{array}{c | c} \bSigma_1 + \bR_{1} \hspace{0.in} & \hspace{0.in} \bR_2\\ \end{array} \right]\bB	\\ 
	\vspace{0.1in}\\
	\bR_2^{\top} \hspace{0.in} &  \hspace{0.in}\bR_3 \hspace{0.in} & \hspace{0.in} \left[\begin{array}{c | c} \bR_2^{\top} \hspace{0.1in} & \hspace{0.1in} \bR_3 \end{array} \right]\bB
	\\ 
	\vspace{0.1in}\\
	\bB^{\top} \left[\begin{array}{c} \bSigma_1 + \bR_1^{\top} \\\\ \bR_2^{\top} \end{array} \right] \hspace{0.1in} \hspace{0.3in} & \hspace{0.in}  \bB^{\top} \left[\begin{array}{c} \bR_2 \\\\ \bR_3^{\top} \end{array} \right]\hspace{0.in}  \hspace{0.in} & \hspace{0.in} \bB^{\top} \left[\begin{array}{c  c} \bSigma_1 + \bR_1^{\top} \hspace{0.2in}  & \hspace{0.2in} \bR_2 \\ \\ \bR_2^{\top} \hspace{0.2in} & \hspace{0.2in} \bR_3^{\top} \end{array} \right]\bB
	\end{array}
	\right]
	\end{array}.
	\]
	\normalsize
	Let 
	\small
	\[
	\bbR_3=
	\left[
	\begin{array}{c c} \bR_3 \hspace{0.3in} & \hspace{0.3in} \left[\begin{array}{c | c} \bR_2^{\top} \hspace{0.1in} & \hspace{0.1in}\bR_3 \end{array} \right]\bB\\
	\vspace{0.1in}\\
	\bB^{\top} \left[\begin{array}{c} \bR_2 \\\\ \bR_3^{\top} \end{array} \right]\hspace{0.1in}  \hspace{0.3in} & \hspace{0.3in}\bB^{\top} \left[\begin{array}{c  c} \bSigma_1 + \bR_1^{\top} \hspace{0.2in}  & \hspace{0.2in} \bR_2 \\ \\ \bR_2^{\top} \hspace{0.2in} & \hspace{0.2in} \bR_3^{\top} \end{array} \right]\bB
	\end{array}
	\right],
	\]
	\normalsize
	and
	$
	\bbR_2= \left[\begin{array}{c  c} \bR_2  & \hspace{0.3in} \left[\begin{array}{c|c} \bSigma_1+\bR_1 &  \, \, \,\bR_2 \end{array} \right]\bB \end{array}
	\right].
	$
	
	Here $\bB \in \mathbb{R}^{k \times (n-k)}$ exists since $\rank(\widetilde{\bSigma}) \leq k$. Moreover, $\widetilde{\bSigma} \succeq 0$ for small enough perturbation.  Therefore, the Schur complement theorem implies   $\bbR_3 \succeq \bbR_2^{\top}(\bSigma_1 + \bR_1)^{-1}\bbR_2$. Thus $\widetilde{\bSigma} \in \mathcal{R}_{\mathcal{M}_{+}}$ requires $\bR_{1}$ to be a symmetric $\mathbb{R}^{r \times r}$ matrix, $\bbR_2$ to be an $\mathbb{R}^{r \times {n-r}}$ matrix, and $\bbR_3$ to be a symmetric $\mathbb{R}^{(n-r) \times (n-r)}$ matrix with $\bbR_3 \succeq \bbR_2^{\top}(\bSigma_1 + \bR_1)^{-1}\bbR_2$. For every small perturbation $\bR_{\delta} = \left[\begin{array}{cc}
	\bR_1 & \bbR_2\\
	\bbR_2^{\top} & \bbR_3\\
	\end{array} \right]$, with $\|\bR_{\delta} \| \leq \delta$, we need to find $\bA_{\epsilon} \in \mathbb{R}^{n \times k}$ such that 
	\begin{equation}\label{Quad_Matrix_Eqn}
	(\bbW + \bA_{\epsilon} )(\bbW + \bA_{\epsilon})^{\top} = \widetilde{\bSigma}\mbox{ or equivalently } \bbW\bA_{\epsilon}^{\top} + \bA_{\epsilon}\bbW^{\top} + \bA_{\epsilon}\bA_{\epsilon}^{\top} = \bR_{\delta}.
	\end{equation}
	Since the last $n-r$ rows of $\bbW$ are all zeros, we obtain
	\small
	\[\bbW\bA_{\epsilon}^{\top} = \left[\begin{array}{c} \bbW_1\\ \bzero \end{array} \right]\left[\begin{array}{c c} (\bA_{\epsilon}^1)^{\top} & (\bA_{\epsilon}^2)^{\top} \end{array} \right]=\left[\begin{array}{c c} \bbW_1(\bA_{\epsilon}^1)^{\top}  & \bbW_1(\bA_{\epsilon}^2)^{\top} \\ \bzero & \bzero \end{array} \right], \quad \mbox{and}\]
	\normalsize
	\[\bA_{\epsilon}\bA_{\epsilon}^{\top} = \left[\begin{array}{c} \bA_{\epsilon}^1 \\ \bA_{\epsilon}^2 \end{array} \right]\left[\begin{array}{c c} (\bA_{\epsilon}^1)^{\top} & (\bA_{\epsilon}^2)^{\top} \end{array} \right]=\left[\begin{array}{c c} \bA_{\epsilon}^1(\bA_{\epsilon}^1)^{\top}  & \bA_{\epsilon}^1(\bA_{\epsilon}^2)^{\top} \\ \bA_{\epsilon}^2(\bA_{\epsilon}^1)^{\top}  & \bA_{\epsilon}^2(\bA_{\epsilon}^2)^{\top} \end{array} \right]. \qquad\]
	where $\bbW_1= \big(\bbW\big)_{1:r,:} \in \mathbb{R}^{r \times k}$ is a full row rank matrix, $\bA_{\epsilon}^1 \in \mathbb{R}^{r \times k}$, and $\bA_{\epsilon}^2 \in \mathbb{R}^{(n-r) \times k}$.
	From Equation (\ref{Quad_Matrix_Eqn}), we get the following three expressions:
	\begin{flalign}
	&\bbW_1(\bA_{\epsilon}^1)^{\top} \, +\, \bA_{\epsilon}^1\bbW_1^{\top} \, + \,  \bA_{\epsilon}^1(\bA_{\epsilon}^1)^{\top} \, =\,  \bR_1,  \label{ex1} \\ 
	&\bbW_1(\bA_{\epsilon}^2)^{\top} \, +\,  \bA_{\epsilon}^1(\bA_{\epsilon}^2)^{\top} \,=\, \bbR_2, \label{ex2}  \\ 
	&\bA_{\epsilon}^2(\bA_{\epsilon}^2)^{\top} \, = \, \bbR_3.   \label{ex3}
	\end{flalign}
	Setting $\bA_{\epsilon}^1 \triangleq \bP(\bbW_1^{\dagger})^{\top}$, where $(\bbW_1)^{\dagger}\triangleq  \bbW_1^{\top}(\bbW_1\bbW_1^{\top})^{-1}$, we obtain 
	\small
	\begin{flalign*}
	&\bbW_1(\bA_{\epsilon}^1)^{\top} + \bA_{\epsilon}^1\bbW_1^{\top} + \bA_{\epsilon}^1(\bA_{\epsilon}^1)^{\top} \\
	=\; & \bbW_1\bbW_1^{\top}\bSigma_1^{-1}\bP^{\top} \, + \, \bP\bSigma_1^{-1}\bbW_1\bbW_1^{\top} \, + \,  \bP\bSigma_1^{-1}\bbW_1\bbW_1^{\top}\bSigma_1^{-1}\bP^{\top}= \bP^{\top}+\bP+\bP\bSigma_1^{-1}\bP^{\top}.
	\end{flalign*}
	\normalsize
	Using Lemma \ref{P_matrix}, we can choose $\delta$ small enough so that for any perturbation matrix $\bR$ with $\|\bR\| < \delta$, there exists a solution $\bP$ with $\|\bP\|= O(\delta)$. More precisely, we can generate $\bP \in \mathbb{R}^{r \times r}$ that satisfies expression (\ref{ex1}), with $\|\bP\|_{\infty} \leq 3\delta$. Also, since $(\bbW_1^{\dagger})^{\top}\bbW_1^{\dagger} = {\bSigma}_1^{-1}$, we obtain 
	$
	\|\big(\bbW_1^{\dagger}\big)_{:,j}\|^2  \leq \dfrac{1}{{\sigma_{min}}} \quad \forall \, j \leq r, 
	$
	where ${\sigma}_{min}$ is the minimum singular value for $\bSigma_1$. Then by definition of $\bA_{\epsilon}^1$, we can bound its norm: 
	\begin{equation}\label{boundAdelta1}
	\|\bA_{\epsilon}^1\| \leq \|\bbW_1^{\dagger}\| \|\bP\| \leq   \dfrac{\sqrt{r}}{\sqrt{{\sigma}_{min}}}3r^2\delta =  \dfrac{3r^{2.5}\delta}{\sqrt{{\sigma}_{min}}}.
	\end{equation}
	Note that $\|\bA_{\epsilon}^1\|$ is of order $\delta$ which can be chosen arbitrarily small so that $\bbW_1+ \bA_{\epsilon}^1$ is full row rank. Define
	\[(\bA_{\epsilon}^2)^{\top} \triangleq (\bbW_1+ \bA_{\epsilon}^1)^{\dagger}\bbR_2 \,+\, \bM , \]
	where
	$
	\bM \subset \{\bM \in \mathbb{R}^{ k \times (n-r)}\, \big| \, \|\bM\| \leq \delta, \, \, \mathcal{C}(\bM) \subset \mathcal{N}\big(\, \bbW_1+ \bA_{\epsilon}^1 \, \big) \},
	$
	and
	\[ (\bbW_1\, +\,  \bA_{\epsilon}^1)^{\dagger} \,\triangleq \, (\bbW_1\, +\,  \bA_{\epsilon}^1)^{\top}[(\bbW_1\, +\,  \bA_{\epsilon}^1)(\bbW_1\, +\,  \bA_{\epsilon}^1)^{\top}]^{-1} \, = \, (\bbW_1\, +\,  \bA_{\epsilon}^1)^{\top}(\bSigma_1 + \bR_1)^{-1}  \]
	with the last equality obtained using (\ref{ex1}).
	Substituting  $\bA_{\epsilon}^2$ in  (\ref{ex2}), we obtain
	\[ (\bbW_1\, +\,  \bA_{\epsilon}^1)(\bA_{\epsilon}^2)^{\top} \,=\, (\bbW_1\, +\,  \bA_{\epsilon}^1)(\bbW_1+ \bA_{\epsilon}^1)^{\dagger}\bbR_2 \,+\, (\bbW_1\, +\,  \bA_{\epsilon}^1)\bM =\bbR_2.\]
	where the last equality is valid since $\mathcal{C}(\bM) \subset \mathcal{N}(\bbW_1+ \bA_{\epsilon}^1)$.
	Substituting $\bA_{\epsilon}^2$ in (\ref{ex3}), we obtain
	\begin{flalign*}
	\bA_{\epsilon}^2(\bA_{\epsilon}^2)^{\top} & = \bar{{\bR}}_2^{\top} (\bSigma_1 + \bR_1)^{-1}(\bSigma_1 + \bR_1)(\bSigma_1 + \bR_1)^{-1} \bar{{\bR}}_2 + \bM^{\top}(\bbW_1\, +\,  \bA_{\epsilon}^1)^{\top}(\bSigma_1 + \bR_1)^{-1} \bar{{\bR}}_2\\
	&+ \bar{{\bR}}_2^{\top}(\bSigma_1 + \bR_1)^{-1}(\bbW_1\, +\,  \bA_{\epsilon}^1)\bM + \bM^{\top}\bM&\\
	&=\bar{{\bR}}_2^{\top} (\bSigma_1 + \bR_1)^{-1}(\bSigma_1 + \bR_1)(\bSigma_1 + \bR_1)^{-1} \bar{{\bR}}_2 + \bM^{\top}\bM\\
	&=\bar{{\bR}}_2^{\top}(\,\bSigma_1 + \bR_1)^{-1}\bar{{\bR}}_2 \,+ \, \bM^{\top}\bM,
	\end{flalign*}
	where the second inequality holds since $\mathcal{C}(\bM) \subset \mathcal{N}(\bbW_1+ \bA_{\epsilon}^1)$.
	Expression (\ref{ex3}) can be satisfied if for any symmetric $\bbR_3 \succeq \bbR_2^{\top}(\,\bSigma_1 \, +\, \bR_1\,)^{-1}\bbR_2$, there exists $\bM$ such that $\bM^{\top}\bM=\bbR_3-\bbR_2^{\top}(\,\bSigma_1 \, +\, \bR_1\,)^{-1}\bbR_2$.
	
	Since $(\bbW_1+ \bA_{\epsilon}^1) \in \mathbb{R}^{r \times k}$ is a full row rank matrix, then dim$\big(\mathcal{N}(\,\bbW_1+\bA_{\epsilon}^1\,)\big)=k-r$.  Let $\bQ \in \mathbb{R}^{k \times (k-r)}$ be a basis for $\mathcal{N}(\bbW_1+ \bA_{\epsilon}^1)$. Then for every 
	$\bM \subset \{\bM \in \mathbb{R}^{ k \times (n-r)}\, \big| \, \|\bM\| \leq \delta, \, \, \mathcal{C}(\bM) \subset \mathcal{\bN}\big(\, \bbW_1+ \bA_{\epsilon}^1 \, \big) \}$,
	there exist $\bN \in \mathbb{R}^{(k-r) \times (n-r)}$ with $\bM=\bQ\bN$, which implies
	$
	\bM^{\top}\bM \,=\, \bN^{\top}\bQ^{\top}\bQ\bN \, = \, \bN^{\top}\bN.
	$
	Since $\bbR_3 - \bbR_2^{\top}(\bSigma_1 + \bR_1)^{-1}\bbR_2$ is the schur complement of $\bSigma + \bR_{\delta}$, then by the Guttman rank additivity formula, we get
	$
	k \geq \rank(\widetilde{\bSigma})=\rank(\bSigma_1 \, + \, \bR_1)+\rank(\bbR_3 - \bbR_2^{\top}(\bSigma_1 + \bR_1)^{-1}\bbR_2),
	$
	which implies
	$
	\rank(\bbR_3 - \bbR_2^{\top}(\bSigma_1 + \bR_1)^{-1}\bbR_2)\leq k-r.
	$
	Thus for any symmetric positive semi-definite matrix   $\bbR_3- \bbR_2^{\top}(\bSigma_1 + \bR_1)^{-1}\bbR_2$, there exist a matrix $\bN \in \mathbb{R}^{(k-r) \times (n-r)}$ such that $\bN^{\top} \bN = \bbR_3 - \bbR_2^{\top}(\bSigma_1 + \bR_1)^{-1}\bbR_2$. It follows that there exist a matrix $\bM \in \mathbb{R}^{k \times (n-r)}$, $\bM \triangleq \bQ\bN$, with 
	$
	\bM^{\top}\bM = \bN^{\top}\bN =\bbR_3 - \bbR_2^{\top}(\bSigma_1 + \bR_1)^{-1}\bbR_2. 
	$
	We have defined $\bA_{\epsilon}=\left[\begin{array}{l} \bA_{\epsilon}^1\\ \bA_{\epsilon}^2 \end{array} \right]$ such that $(\bbW+\bA_{\epsilon})(\bbW+\bA_{\epsilon})^{\top} = \widetilde{\bSigma}$.
	
	We now obtain an upper-bound on $\|\bA_{\epsilon}\|$. Since 
	$
	\big((\bbW_1\, +\,  \bA_{\epsilon}^1)^{\dagger}\big)^{\top}  (\bbW_1\, +\,  \bA_{\epsilon}^1)^{\dagger}\,=\, (\bSigma_1 + \bR_1)^{-1}(\bbW_1\, +\,  \bA_{\epsilon}^1)(\bbW_1\, +\,  \bA_{\epsilon}^1)^{\top} (\bSigma_1 + \bR_1)^{-1} = (\bSigma_1 + \bR_1)^{-1},
	$
	we obtain   
	\[\|\big((\bbW_1\, +\,  \bA_{\epsilon}^1)^{\dagger}\big)_{:,j}\|^2 \leq \dfrac{1}{{\sigma}_{min} -\delta} \quad \forall \, j \leq r.\]
	Then by the definition of $\bA_{\epsilon}^2$, we can bound its norm as follows
	\begin{equation}\label{boundAdelta2}
	\|\bA_{\epsilon}^2 \| \leq \|(\bbW_1\, +\,  \bA_{\epsilon}^1)^{\dagger}\|\|\bbR_2\| + \|\bM\| \leq \dfrac{\delta \sqrt{r}}{\sqrt{{\sigma}_{min} - \delta}}  +\delta.
	\end{equation}
	Using (\ref{boundAdelta1}) and (\ref{boundAdelta2}), we obtain
	\begin{flalign*}
	\|\bA_{\epsilon} \| & \leq \dfrac{3r^{2.5}\delta}{\sqrt{{\sigma}_{min}}}+\dfrac{\delta\sqrt{r}}{\sqrt{{\sigma}_{min} - \delta}}  +\delta \leq \dfrac{3r^{2.5}\delta}{\sqrt{{\sigma}_{min}}} +\dfrac{\delta\sqrt{2r}}{\sqrt{{\sigma}_{min}}}  +\delta\\
	& \leq\delta \, \, \dfrac{3r^{2.5} + \sqrt{2r}+ \sqrt{\sigma_{min}}}{\sqrt{{\sigma}_{min}}},
	\end{flalign*}
	where the second inequality assumes $\delta \leq \sigma_{min} /2$. Now, for a given $\epsilon > 0$, choose
	\[
	\delta \leq \min \Bigg\{ \dfrac{\epsilon\sqrt{{\sigma}_{min}}}{3r^{2.5} + \sqrt{2r}+ \sqrt{\sigma_{min}}} \, , \, \sigma_{min} /2 \Bigg\}.  
	\]
	This choice of $\delta$ leads to $\|\bA_{\epsilon} \|  \leq \epsilon $, which completes the proof.
 
 \end{proof}

\section{Proof of the Theorem \ref{thm: Two-layer} }\label{two-layer-app}
\begin{proof}
The proof for the degenerate case is done by constructing a descent direction if the point is critical but not global. Let $(\bbW_2$, $\bbW_1)$ be a degenerate critical point, i.e., rank($\bbW_2\bbW_1)< \min\{d_2,d_1,d_0\}$. Then, based on the dimensions of $d_0$, $d_1$, and $d_2$, we have one of the following cases:
	\begin{itemize}
		\item $d_2 < d_1 \mbox{ then } \exists \, \, \bb \neq \bzero \mbox{ such that } \bb \in \mathcal{N} \big(\bbW_2\big)$.
		\item $d_0 < d_1 \mbox{ then } \exists \, \, \bb \neq \bzero \mbox{ such that } \bb \in \mathcal{N} \big(\bbW_1^{\top}\big)$.
		\item $d_1 \leq d_2, \mbox{and } d_1 \leq d_0 \mbox{ then either } \bbW_2 \textrm{ is rank deficient and } \exists \, \, \bb \neq \bzero\;  \st \; \bb \in \mathcal{N} \big(\bbW_2\big) \mbox { or }\\
		 \bbW_1 \mbox{ is rank deficient and } \exists \, \, \bb \neq \bzero \mbox{ such that } \bb \in \mathcal{N} \big(\bbW_1^{\top}\big)$.
	\end{itemize}
	So in all cases either ${\cal N} \big( \bbW_2 \big) \neq \emptyset$ or ${\cal N} \big( \bbW_1^{\top} \big)\neq \emptyset$. Also, let ${\bDelta}=\bbW_2\bbW_1 \bX - \bY $. If  ${\bDelta}{\bX}^{\top}=\bzero$, then by convexity of the square loss error function, the point $(\bbW_2,\bbW_1)$ is a global minimum of (\ref{Two-layer-SE-loss-2}). Else, there exists $(i,j)$ such that $\big\langle  {\bX}_{i,:},{\bDelta}_{j,:} \big\rangle \neq 0$. We now use first and second order optimality conditions to construct a descent direction when the current critical point is not global.
	
	\noindent{\it First order optimality condition:} By considering perturbations in the directions $\bA \in \mathbb{R}^{d_2 \times d_1}$ and $\bB \in \mathbb{R}^{d_1 \times d_0}$ for the optimization problem
	\begin{equation}
	\min_{t} \, \, \dfrac{1}{2} \|({\bbW}_2 + t\bA)({\bbW}_1 + t\bB) \bX \, - \, \bY \, \|^2,
	\end{equation}
	we obtain the first order optimality condition:
	\[\big\langle \, \bA\bbW_1\bX \,+\, \bbW_2 \bB \bX \,,\, \bDelta \, \big\rangle =0, \quad \forall \bA \in \mathbb{R}^{d_2 \times d_1}, \, \bB \in \mathbb{R}^{d_1 \times d_0}.\]
	
	\noindent{\it and second order optimality condition:}
	\[2\,\big\langle \, \bA \bB \bX,\, \bDelta \, \big\rangle  + \|\bA \bbW_1\bX + \bbW_2 \bB \bX\|^2 \,  \geq 0 \quad  \forall \bA \in \mathbb{R}^{d_2 \times d_1}, \, \bB \in \mathbb{R}^{d_1 \times d_0}.\]
	Suppose $(\bbW_2,\bbW_1)$ is a critical point and there exists $\bb \neq 0$, $\bb \in \mathcal{N}(\bbW_2)$.  Define 
	\[{\bB}_{:,l} \triangleq \begin{cases}
	\begin{array}{ll}
	{\alpha}\bb \qquad & \quad \mbox{if } l=i,\\
	\bzero \qquad & \quad \mbox{otherwise}
	\end{array}
	\end{cases} \qquad
	{\bA}_{l,:} \triangleq \begin{cases}
	\begin{array}{ll}
	{\bb}^{\top} \qquad & \quad \mbox{if } l=j,\\
	\bzero \qquad & \quad \mbox{otherwise}
	\end{array}
	\end{cases}\]
	where $\alpha$ is a scalar constant.
	Then, using the second order optimality condition, for $\bc = \|\bA \bbW_1\bX \|^2$, we get
	$2\alpha \underbrace{\|{\bb}\|^2}_{\neq 0} \underbrace{\big\langle {\bX}_{i,:}, \bDelta_{j,:} \big \rangle}_{\neq 0} + \, c \geq 0.$
	Since this is true for every value of $\alpha$, ${\bb}$ should be zero which contradicts the assumption on the choice of $\bb$. Hence ${\cal N}(\bbW_2)=\emptyset$.
	Similarly, when  $(\bbW_2,\bbW_1)$ is a critical point and there exists $\ba^{\top} \neq 0$, $\ba^{\top} \in \mathcal{N}(\bbW_1^{\top})$, we can show that $(\bbW_2,\bbW_1)$ is a second order saddle point of  (\ref{Two-layer-SE-loss-2}). Combining these results, we get that every degenerate critical point that is not a global optimum is a second-order saddle point.
	
	We now show the result for the non-degenerate case. Let $(\bbW_2$, $\bbW_1)$ be a non-degenerate local minimum, i.e., rank($\bbW_2\bbW_1)=\min\{d_2,d_1,d_0\}$. {\color{black} It follows by Theorem~\ref{thm:Loc_Open}, that the matrix product is locally open at $\bbW_2\bbW_1$.} Then by Observation \ref{lm:LocalMinMapping}, $\bZ=\bbW_2\bbW_1$ is a local optimum of problem (\ref{Two-layer-SE-loss-formulated}) which is in fact global by Lemma \ref{lm:Relax-X-Assumption}. \end{proof}

\subsection{Proof of Corollary~\ref{cor: degenerate2layer}}\label{App:corollary_degenerate2layer}
\begin{proof}
	We follow the same steps used in the proof of Theorem \ref{thm: Two-layer} to show the result. Similar to the proof of Theorem~\ref{thm: Two-layer}, we obtain the following first and second order optimality conditions:
\[\langle \bA \bbW_1\bX + \bbW_2 \bB \bX, \nabla \ell(\bbW_2\bbW_1\bX - \bY) \rangle =0 \, \quad \, \forall \bA \in \mathbb{R}^{d_2 \times d_1}, \, \bB \in \mathbb{R}^{d_1 \times d_0}\]
\[2\langle \, \bA \bB \bX,\, \nabla \ell(\bbW_2\bbW_1\bX - \bY) \, \rangle  + h\big(\bA \bbW_1\bX, \, \bbW_2 \bB \bX ,\bbW_2\bbW_1\bX)  \geq 0 \, \,\,  \forall \bA \in \mathbb{R}^{d_2 \times d_1}, \, \bB \in \mathbb{R}^{d_1 \times d_0},\]
	where $h(\cdot)$ is a function that has a tensor representation. But we only need to know that it is a function of $\bA\bbW_1\bX, \, \bbW_2\bB\bX$, and $\bbW_2\bbW_1\bX$. If ${\nabla} \ell(\bbW_2\bbW_1\bX -\bY){\bX}^{\top}= \bzero$, then by convexity of $\ell(\cdot)$, $(\bbW_2, \bbW_1)$ is a global minimum. Otherwise, there exists $(i,j)$ such that $\big\langle  {\bX}_{i,:},\big({\nabla} \ell(\bbW_2\bbW_1\bX-\bY)\big)_{j,:} \big\rangle \neq 0$. Using the same former argument in proof of Theorem \ref{thm: Two-layer}, we choose $\bA$ and $\bB$ such that $h(\bA\bbW_1\bX, \, \bbW_2\bB\bX, \, \bbW_2\bbW_1\bX)$ is some constant that does not depend on $\alpha$, and 
	\[\big \langle \, \bA\bB\bX, \nabla \ell(\bbW_2\bbW_1\bX-\bY) \, \rangle =\alpha \underbrace{\big\langle  {\bX}_{i,:},\big({\nabla} \ell(\bbW_2\bbW_1\bX - \bY)\big)_{j,:} \big\rangle}_{\neq 0}.
	\] 
	Then by proper choice of $\alpha$ we show that the point $(\bbW_2,\bbW_1)$ is a second order saddle point.\end{proof}


\section{Proof of Theorem \ref{thm:Deep-linear-Result}}\label{Deep-linear-result-app}
Consider the training problem of a multi-layer deep linear neural network: 
\begin{equation}\label{Linear_SEloss-3}
\displaystyle{ \min_{\bW}} \, \, \dfrac{1}{2} \|\bW_h \cdots \bW_1\bX \, - \, \bY \, \|^2.
\end{equation}
Here  $\bW = \big( \bW_i \big)_{i=1}^h$, $\bW_i \in \mathbb{R}^{d_i \times d_{i-1}}$ are the weight matrices, $\bX \in \mathbb{R}^{d_0 \times n}$ is the input training data, and $\bY \in \mathbb{R}^{d_h \times n}$ is the target training data. Based on our general framework, the corresponding auxiliary optimization problem is given by 
\begin{equation} \label{Transformed_Linear_SEloss-3}
\begin{array}{ll}
\displaystyle{ \operatornamewithlimits{\mbox{minimum}}_{\bZ \in \mathbb{R}^{d_h \times d_0}}} & \dfrac{1}{2} ||\bZ\bX \, - \, \bY||^2 \\
\mbox{subject to} & \rank(\bZ) \leq d_p \triangleq \min_{0 \leq i\leq h} \,d_i
\end{array}.
\end{equation}

\begin{lemma}\label{lm:degenerate-local-global}
	Consider a degenerate critical point $\bbW = (\bbW_h, \ldots , \bbW_1)$ with \\
	${\cal N}( \bbW_i)$ and ${\cal N}\big(  \bbW_i^{\top}  \big)$ for $h-1 \leq i \leq 2$ all non-empty. If 
	\[{\cal N}\big( \, \bbW_h \,\big) \mbox{ is non-empty} \quad \mbox{or} \quad {\cal N}\big( \, \bbW_1^{\top} \,\big) \mbox{ is non-empty},
	\]
	then $\bbW$ is either a global minimum or a saddle point of problem (\ref{Linear_SEloss-3}).
\end{lemma}


\begin{proof}
	Suppose that ${\cal N}\big( \, \bbW_h \,\big)$ is non-empty. Let $\bDelta=\bbW_h \cdots \bbW_1\bX-\bY$. If $\bDelta\bX^{\top} = \bzero$,  by convexity of the  loss  function, the point $\bbW=(\bbW_h, \ldots ,\bbW_1)$ is a global minimum of (\ref{Linear_SEloss-3}). Else, there exist $(i,j)$ such that $\big\langle \, \bX_{i,:}, \, \bDelta_{j , :} \, \big\rangle \, \neq 0$. We define the set ${\cal K} \triangleq \{ k \in \mathbb{N}\, | \, 3 \leq k \leq h, \quad {\cal N}(\bbW_k) \, \perp {\cal N}\big((\bbW_{k-1}\bbW_{k-2} \cdots \bbW_{2})^{\top}\big)\}$. We split the rest of the proof into two cases that correspond to ${\cal K}$ being empty and non-empty. 
	
	\noindent{\bf Case a:}
	Assume ${\cal K}$ is non-empty. We define $k^{*} \triangleq \displaystyle{\operatornamewithlimits{\mbox{maximum}}_{k  \in {\cal K}}} \, \, k $. 
	By definition of the set ${\cal K}$ and choice of $k^{*}$, the null space ${\cal N} \big( \,\bbW_{k^{*}} \big)$ is orthogonal to the null-space ${\cal N} \big(\, (\bbW_{k^{*}-1} \cdots \bbW_2)^{\top} \big)$. This implies there exists a non-zero $\bb \in \mathbb{R}^{d_{k^{*}-1}}$ such that $\bb \in {\cal N} \big( \,\bbW_{k^{*}} \, \big) \, \cap \, {\cal C} \big(\, \bbW_{k^{*}-1} \cdots \bbW_2 \, \big)$. By considering perturbation in directions $\bA = (\bA_h, \ldots, \bA_1)$, $\bA_i \in \mathbb{R}^{d_i \times d_{i-1}}$ for the optimization problem
	\begin{equation}\label{g(t)}
	\displaystyle{ \min_{t}} \, \,g(t) \triangleq \dfrac{1}{2} \|(\bbW_h + t\bA_h)\cdots(\bbW_1+ t\bA_1)\bX \, - \, \bY \, \|^2,
	\end{equation}
	we examine the optimality conditions for a specific direction $\bbA$.
	
	Let
	\[(\bbA_h)_{l,:} \triangleq  \begin{cases}
	{\alpha}_h \bp_h^{\top} \quad \mbox{if } l=j,\\
	\bzero \qquad \quad \mbox{otherwise}\\
	\end{cases} \quad
	(\bbA_1)_{:,l} \triangleq \begin{cases}
	{\alpha}_1 \bb_1 \quad \mbox{if } l=i,\\
	\bzero \qquad  \, \, \, \mbox{otherwise}\\
	\end{cases}\]
	\[
	\bbA_k  \triangleq \begin{cases}
	\begin{array}{ll}
	\bb_k\bp_k^{\top} \quad & \mbox{if } k^{*} + 1 \leq k \leq h-1\\
	\bb_k\bb^{\top} \quad & \mbox{if }  k = k^{*}\\
	\bzero						 & \mbox{if } 2 \leq k \leq  k^{*}-1,
	\end{array}
	\end{cases} 
	\]
	where ${\alpha}_h$ and ${\alpha}_1$ are scalar constants, $\bb_1 \in \mathbb{R}^{d_1}$ such that $\bbW_{k^{*}-1}\cdots \bbW_{2}\bb_1 = \bb$, and
	\begin{equation}\label{bkpk}
	\bp_k \in {\cal N}\big( (\bbW_{k-1} \cdots \bbW_2)^{\top}\big), \; \bb_{k-1} \in {\cal N}\big( \bbW_{k}\big), \; \mbox{and } \langle \, \bp_k, \, \bb_{k-1} \, \rangle \neq 0 \; \forall \, \, k^{*}+1 \leq k \leq h.
	\end{equation} 
	Notice that such $\bp_k$ and $\bb_{k-1}$ exist from the definition of ${\cal K}$ and choice of $k^{*}$.
	For this particular choice of $\bbA = (\bbA_h, \ldots , \bbA_1)$, we obtain 
	\begin{equation}\label{zero f}
	\bbW_{k+1}\bbA_k =\bzero  \mbox{ for } k^{*} \leq k \leq h-1; \; \mbox{and} \; \bbA_k\bbW_{k-1}\cdots \bbW_{2} = \bzero \; \mbox{for } k^{*}+1 \leq k \leq h.
	\end{equation}
	We now show that $(\bbA_{h}, \ldots , \bbA_{1})$ is in fact a descent direction. Before proceeding, let us define some notation to ease the expressions of the optimality conditions. Let ${\cal V}$ be an index set that is a subset of $\{1, \ldots ,h\}$. We define the function $f(\bbA^{\cal V}, \bbW^{-{\cal V}})$ which is the matrix product attained from $\bbW_h \cdots  \bbW_1\bX$ by replacing matrices $\bbW_{v}$ by matrices $\bbA_{v}$ for every $v \in {\cal V}$. For instance, if $h=5$ and ${\cal V}=\{2,3,5\}$, then $f(\bbA^{\cal V}, \bbW^{-{\cal V}}) = \bbA_5\bbW_4\bbA_3\bbA_2\bbW_1\bX$. We now determine index sets ${\cal V}$, with $|{\cal V}|\geq 1$, that correspond to non-zero $f(\bbA^{\cal V} , \bbW^{-{\cal V}} )$. First note by definition of $\bbA$, if ${\cal V} \, \cap \, \{k^{*}-1, \ldots , 2\} \neq \emptyset$, then $f(\bbA^{\cal V} , \bbW^{-{\cal V}} )=\bzero$. Also by (\ref{zero f}), for any $k^{*} \leq v \leq h-1$, if $v \in {\cal V}$ then either $\{k^{*}, \ldots , h\} \in {\cal V}$ or $f(\bbA^{\cal V} , \bbW^{-{\cal V}})=\bzero$. This implies that $\bbA_h\cdots \bbA_{k^{*}}\bbW_{k^{*}-1}\cdots\bbW_1\bX$ and $\bbA_h\cdots \bbA_{k^{*}}\bbW_{k^{*}-1}\cdots \bbW_{2}\bbA_1\bX$ are the only terms that can take non-zero values. Using the definition equation (\ref{g(t)}) we obtain
	\[g(t) = \dfrac{1}{2}\|t^{h-k^{*} +1}\bbA_h\cdots \bbA_{k^{*}}\bbW_{k^{*}-1}\cdots\bbW_1\bX + t^{h-k^{*} + 2}\bbA_h\cdots \bbA_{k^{*}}\bbW_{k^{*}-1}\cdots \bbW_{2}\bbA_1\bX + \bDelta\|^2.\]
	
	It follows that
	$
	\dfrac{{\partial}^r g(t)}{{\partial} t^r} \Bigg|_{t=0} = 0 \quad \mbox{for all } r \leq h-k^{*}
	$
	and
	\[\dfrac{{\partial}^{h-k^{*}+1} g(t)}{{\partial} t^{h-k^{*}+1}} \Bigg|_{t=0} = c_1\big\langle \, \bbA_h\cdots \bbA_{k^{*}}\bbW_{k^{*}-1}\cdots\bbW_1\bX, \bDelta \, \big\rangle  
	,\]
	where $c_1 >0$ is a scalar. If $\big\langle \, \bbA_h\cdots \bbA_{k^{*}}\bbW_{k^{*}-1}\cdots\bbW_1\bX, \bDelta \, \big\rangle \neq 0$, then by properly choosing the sign of ${\alpha}_h$ such that\\ $\big\langle \, \bbA_h\cdots \bbA_{k^{*}}\bbW_{k^{*}-1}\cdots\bbW_1\bX, \bDelta \, \big\rangle < 0$, we get a descent direction. Otherwise, 
	\begin{flalign*}
	\dfrac{{\partial}^{h-k^{*}+2} g(t)}{{\partial} t^{h-k^{*}+2}} \Bigg|_{t=0} = c_1\big\langle \, \bbA_h\cdots \bbA_{k^{*}}\bbW_{k^{*}-1}\cdots\bbW_2 \bbA_1\bX, \bDelta \, \big\rangle + h(\bbA_h\cdots \bbA_{k^{*}}\bbW_{k^{*}-1}\cdots\bbW_1\bX).
	\end{flalign*}
	
	where $c_1 > 0$ is a scalar, and $h(\cdot)$ is a function of $\bbA_h\cdots \bbA_{k^{*}}\bbW_{k^{*}-1}\cdots\bbW_1\bX$.\\
	We now evaluate the term $\big\langle \, \bbA_h\cdots \bbA_{k^{*}}\bbW_{k^{*}-1}\cdots \bbW_2\bbA_1\bX, \bDelta \, \big\rangle$. Since $(\bbA_h)_{l,:} = \bzero$ for all $l \neq j$ and $(\bbA_1)_{:,l} = \bzero$ for all $l \neq i$, we only need to compute the $(j,i)$ index $\big( \bbA_h\cdots \bbA_{k^{*}}\bbW_{k^{*}-1}\cdots \bbW_2\bbA_1\big)_{(j,i)}$ as all other indices are zero. For some constant $c=\bp_h^{\top}\bb_{h-1}\bp_{h-1}^{\top}\bb_{h-2}\cdots\bp_{k^{*}+1}^{\top}\bb_{k^{*}}\bb^{\top}\bb$, we obtain
	\begin{flalign*}
	&\;c_1\big( \bbA_h\cdots \bbA_{k^{*}}\bbW_{k^{*}-1}\cdots \bbW_2\bbA_1\big)_{(j,i)}\\
	= & \;c_1{\alpha}_h{\alpha}_1 \bp_h^{\top}\bb_{h-1}\bp_{h-1}^{\top}\bb_{h-2}\cdots\bp_{k^{*}+1}^{\top}\bb_{k^{*}}\bb^{\top}\bbW_{k^{*}-1}\cdots \bbW_{2}\bb_1\\
	= & \;c_1{\alpha}_h{\alpha}_1 \bp_h^{\top}\bb_{h-1}\bp_{h-1}^{\top}\bb_{h-2}\cdots\bp_{k^{*}+1}^{\top}\bb_{k^{*}}\bb^{\top}\bb= {\alpha}_h{\alpha}_1c ,
	\end{flalign*}
	where $c$ is non-zero by our choice of $\bb$, $\bp_k$ and $\bb_{k-1}$ for $k^{*}+1 \leq k \leq h$ as defined in (\ref{bkpk}). For a fixed ${\alpha}_h \neq 0$, $h(\bbA_h\cdots \bbA_{k^{*}}\bbW_{k^{*}-1}\cdots\bbW_1\bX)$ is a constant scalar we denote by $c_{\alpha}$. Then by properly choosing ${\alpha}_1$ such that
	$
	\underbrace{{\alpha}_h}_{\neq 0}{\alpha}_1\underbrace{c}_{\neq 0} \underbrace{\big\langle \, \bX_{i,:}, \bDelta_{j,:} \, \big\rangle}_{\neq 0} + c_{\alpha} < 0,
	$
	we get a descent direction. This completes the first case.
	
	\noindent{\bf Case b:}
	Assume ${\cal K}$ is empty. We consider
	\[(\bbA_h)_{l,:}\triangleq \begin{cases}
	\begin{array}{l l}
	{\alpha}_h \bp_h^{\top} \quad & \mbox{if } l=j,\\
	\bzero & \mbox{otherwise}\\
	\end{array}\end{cases} \quad
	(\bbA_1)_{:,l}\triangleq \begin{cases}
	\begin{array}{l l}
	{\alpha}_1 \bb_1 \quad &\mbox{if } l=i,\\
	\bzero & \mbox{otherwise}\\
	\end{array} \end{cases}\]
	\[
	\bbA_k  \triangleq \begin{cases}
	\begin{array}{ll}
	\bb_k\bp_k^{\top} \quad & \mbox{if } 3 \leq k \leq h-1\\
	\bb_k\bb_1^{\top} \quad & \mbox{if }  k = 2,
	\end{array}
	\end{cases} 
	\]
	where ${\alpha}_h$ and ${\alpha}_1$ are scalar constants, $\bb_1 \in {\cal N}\big( \bbW_2 \big)$, and
	\begin{equation}\label{bkpk2}
	\bp_k \in {\cal N}\big( (\bbW_{k-1} \cdots \bbW_2)^{\top}\big), \; \bb_{k-1} \in {\cal N}\big( \bbW_{k}\big), \; \mbox{and } \langle \, \bp_k, \, \bb_{k-1} \, \rangle \neq 0, \; \forall \, \, 3 \leq k \leq h.
	\end{equation} 
	For this particular choice of $\bbA = (\bbA_h, \ldots , \bbA_1)$, we obtain 
	$
	\bbW_{k+1}\bbA_k =0$ for $2 \leq k \leq h-1$; and $\bbA_k\bbW_{k-1}\cdots \bbW_{2} =0$ for $3 \leq k \leq h.$
	We now determine index sets ${\cal V}$, with $\bigl| {\cal V} \bigr| \geq 1$, that correspond to non-zero $f(\bbA^{\cal V} , \bbW^{-{\cal V}} )$. By (\ref{bkpk2}), for any $2 \leq v \leq h-1$, if $v \in {\cal V}$ then either $\{2, \ldots , h\} \in {\cal V}$ or $f(\bbA^{\cal V} , \bbW^{-{\cal V}})=\mathbf{0}$. This directly imply that 
	$\bbA_h\cdots \bbA_{2}\bbW_{1}\bX$ and $\bbA_h\cdots \bbA_{1}\bX$ are the only terms that can take non-zero values. Using the definition of equation (\ref{g(t)}) we obtain
	
	$
	g(t) = \dfrac{1}{2}\|t^{h-1}\bbA_h\cdots \bbA_{2}\bbW_{1}\bX + t^{h}\bbA_h\cdots \bbA_{1}\bX + \bDelta\|^2.
	$
	It follows that
	$
	\dfrac{{\partial}^r g(t)}{{\partial} t^r} \Bigg|_{t=0} = 0 \quad \mbox{for all } r \leq h-2,
	$
	and
	$
	\dfrac{{\partial}^{h-1} g(t)}{{\partial} t^{h-1}} \Bigg|_{t=0} = c_1\big\langle \, \bbA_h\cdots \bbA_{2}\bbW_1\bX, \bDelta \, \big\rangle  
	,$
	where $c_1>0$ is a scalar. If $\big\langle \, \bbA_h\cdots \bbA_{2}\bbW_1\bX, \bDelta \, \big\rangle \neq 0$, then by properly choosing the sign of ${\alpha}_h$ such that $\big\langle \, \bbA_h \cdots \bbA_{2}\bbW_1\bX, \bDelta \, \big\rangle < 0$, we get a descent direction. Otherwise,
	\[
	\dfrac{{\partial}^{h} g(t)}{{\partial} t^{h}} \Bigg|_{t=0} = c_1\big\langle \, \bbA_h\cdots \bbA_{1}\bX, \bDelta \, \big\rangle + h(\bbA_h\cdots \bbA_{2}\bbW_{1}\bX)  
	,\]
	where $c_1 >0$ is a scalar, and $h(\cdot)$ is a function of $\bbA_h\cdots \bbA_{2}\bbW_1\bX$. We now evaluate the term $\big\langle \, \bbA_h\cdots \bbA_{1}\bX, \bDelta \, \big\rangle$. Since $(\bbA_h)_{l,:} = \bzero$ for all $l \neq j$ and $(\bbA_1)_{:,l} = \bzero$ for all $l \neq i$, we only need to compute the $(j,i)$ index $\big( \bbA_h\cdots \bbA_{1}\big)_{(j,i)}$ as all other indices are zero. For some constant 
	\[c=\bp_h^{\top}\bb_{h-1}\bp_{h-1}^{\top}\bb_{h-2}\cdots\bp_{3}^{\top}\bb_{2}\bb_1^{\top}\bb_1,\]
	we obtain
	\begin{flalign*}
	c_1\big( \bbA_h\cdots \bbA_{1}\big)_{(j,i)}& = c_1{\alpha}_h{\alpha}_1 \bp_h^{\top}\bb_{h-1}\bp_{h-1}^{\top}\bb_{h-2}\cdots\bp_{3}^{\top}\bb_{2}\bb_1^{\top}\bb_1 = {\alpha}_h{\alpha}_1 c,
	\end{flalign*}
	where $c$ is non-zero by our choice of $\bb$, $\bp_k$ and $\bb_{k-1}$ for $3 \leq k \leq h$ as defined in (\ref{bkpk2}). For a fixed ${\alpha}_h \neq 0$, $h(\bbA_h\cdots \bbA_{2}\bbW_1\bX)$ is a constant scalar we denote by $c_{\alpha}$. Then by properly choosing $\alpha_1$ such that 
	\[\underbrace{{\alpha}_h}_{\neq 0}{\alpha}_1\underbrace{c}_{\neq 0} \underbrace{\big\langle \, \bX_{i,:}, \bDelta_{j,:} \, \big\rangle}_{\neq 0} + c_{\alpha} < 0,\]
	we get a descent direction. This completes the second case.
	
	Now if ${\cal N}\big( \, \bbW_1^{\top} \,\big)$ is non-empty, we define the set 
	\[{\cal K} \triangleq \{ k \, | \, 1 \leq k \leq h-2, \, \, {\cal N}(\bbW_{h-1}\cdots \bbW_{k+1}) \, \perp \, {\cal N}(\bbW_k^{\top})\},\]
	and use a similar proof scheme to show the result. More specifically, we split the proof into two cases that correspond to ${\cal K}$ being empty and non-empty.\\
	
	\noindent{\bf Case a:}
	Assume ${\cal K}$ is non-empty. We define $k^{*} \triangleq \displaystyle{\operatornamewithlimits{\mbox{minimum}}_{k  \in {\cal K}}} \, \, k $. 
	By definition of the set ${\cal K}$ and choice of $k^{*}$, the null space ${\cal N} \big( \,\bbW_{k^{*}}^{\top} \big)$ is orthogonal to the null-space ${\cal N} \big(\, \bbW_{h-1} \cdots \bbW_{k^{*}+1} \big)$. This implies there exists a non-zero $\bp \in \mathbb{R}^{d_{k^{*}}}$ such that $\bp \in {\cal N} \big( \,\bbW_{k^{*}}^{\top} \, \big) \, \cap \, {\cal C} \big(\, (\bbW_{h-1} \cdots \bbW_{k^{*}+1})^{\top} \big)$. By considering perturbation in directions $\bA = (\bA_h, \ldots, \bA_1)$, $\bA_i \in \mathbb{R}^{d_i \times d_{i-1}}$ for the optimization problem
	\begin{equation}
	\displaystyle{\min_{t}} \, \,g(t) \triangleq \dfrac{1}{2} \|(\bbW_h + t\bA_h)\cdots(\bbW_1+ t\bA_1)\bX \, - \, \bY \, \|^2,
	\end{equation}
	we examine the optimality conditions for a specific direction $\bbA$.
	
	Let
	\[(\bbA_h)_{l,:} \triangleq \begin{cases}
	{\alpha}_h \bp_h^{\top} \quad \mbox{if } l=j,\\
	\bzero \qquad \quad \mbox{otherwise}\\
	\end{cases} \quad
	(\bbA_1)_{:,l} \triangleq \begin{cases}
	{\alpha}_1 \bb_1 \quad \mbox{if } l=i,\\
	\bzero \qquad  \, \, \, \mbox{otherwise}\\
	\end{cases}\]
	\[
	\bbA_k  \triangleq \begin{cases}
	\begin{array}{ll}
	\bb_k\bp_k^{\top} \quad & \mbox{if } 2 \leq k \leq k^{*}-1\\
	\bp\bp_k^{\top} \quad & \mbox{if }  k = k^{*}\\
	\bzero						 & \mbox{if } k^{*} +1 \leq k \leq  h-1,
	\end{array}
	\end{cases} 
	\]
	where ${\alpha}_h$ and ${\alpha}_1$ are  constants and $\bp_h \in \mathbb{R}^{d_{h-1}}$ with  
	\[\bp_h^{\top}\bbW_{h-1}\cdots \bbW_{k^{*}+1} = \bp^{\top}, \quad 
	\bp_k \in {\cal N}\big( \bbW_{k-1}^{\top} \big), \quad \bb_{k-1} \in {\cal N}\big( \bbW_{h-1}\cdots \bbW_{k} \big), \quad \mbox{and } \langle \, \bp_k, \, \bb_{k-1} \, \rangle \neq 0
	\]
	for all $ 2 \leq k \leq k^{*}$.
	Notice that such $\bp_k$ and $\bb_{k-1}$ exist from the definition of ${\cal K}$ and choice of $k^{*}$.
	For this particular choice of $\bbA = (\bbA_h, \ldots , \bbA_1)$, we obtain 
	\begin{equation}
	\bbA_k\bbW_{k-1} =\bzero \; \mbox{for} \quad 2 \leq k \leq k^{*}; \quad \mbox{and} \quad \bbW_{h-1}\cdots \bbW_{k+1}\bbA_{k} = \bzero \; \mbox{for } 1 \leq k \leq k^{*} - 1.
	\end{equation}
	The same argument used above can be used to show that $(\bbA_{h}, \ldots , \bbA_{1})$ is actually a descent direction. This completes the proof of the first case.
	
	\noindent{\bf Case b:}
	Assume ${\cal K}$ is empty. We consider
	\[(\bbA_h)_{l,:} \triangleq \begin{cases}
	{\alpha}_h \bp_h^{\top} \quad \mbox{if } l=j,\\
	\bzero \qquad \quad \mbox{otherwise}\\
	\end{cases} \quad
	(\bbA_1)_{:,l} \triangleq \begin{cases}
	{\alpha}_1 \bb_1 \quad \mbox{if } l=i,\\
	\bzero \qquad  \, \, \, \mbox{otherwise}\\
	\end{cases}\]
	\[ 
	\bbA_k  \triangleq \begin{cases}
	\begin{array}{ll}
	\bb_k\bp_k^{\top} \quad & \mbox{if } 2 \leq k \leq h-2\\
	\bp_h\bp_k^{\top} \quad & \mbox{if }  k = h-1,
	\end{array}
	\end{cases} 
	\]
	where ${\alpha}_h$ and ${\alpha}_1$ are scalar constants, $\bp_h \in {\cal N}\big( \bbW_{h-1}^{\top} \big)$, and
	$
	\bp_k \in {\cal N}\big( \bbW_{k-1}^{\top} \big), \; \bb_{k-1} \in {\cal N}\big( \bbW_{h-1}\cdots \bbW_{k} \big),$ and $\langle \, \bp_k, \, \bb_{k-1} \, \rangle \neq 0 \; \forall \, \, 2 \leq k \leq h-1.
	$
	For this particular choice of $\bbA $, we obtain 
	$
	\bbA_k\bbW_{k-1} =\bzero \quad \mbox{for } 2 \leq k \leq h - 1\quad \mbox{and} \quad \bbW_{h-1}\cdots \bbW_{k+1}\bbA_{k} = \bzero \quad \mbox{for } 1 \leq k \leq h-2.
	$
	The same argument used above can be used to show that $(\bbA_{h}, \ldots , \bbA_{1})$ is actually a descent direction. This completes the second case and thus completes the proof.
	\end{proof}
	
	Following the same steps of the proof in Lemma \ref{lm:degenerate-local-global}, we get the same result when replacing the square loss error by a general convex and differentiable function $\ell(\cdot)$. We are now ready to prove the main result restated below.\\

\noindent\textbf{Theorem~\ref{thm:Deep-linear-Result}}
If there does not exist $p_1$ and $p_2$, $1 \leq p_1 < p_2 \leq h-1$ with $d_h > d_{p_2}$ and $d_0 > d_{p_1}$, then every local minimum of problem~(16) is a global minimum.

\begin{proof}
	We now show that if such a pair $\{p_2, p_1\}$ does not exist, then if $\bbW$ is not a global minimum, we can construct a descent direction.
	
	First notice that if for some $1 \leq i \leq h-1$, $\bbW_i$ is full column rank, then using Proposition 1, ${\cal M}_{i+1,i}(\cdot)$ is locally open at $(\bbW_{i+1},\bbW_i)$ and $\bbW_{i+1}\bbW_i \in \mathbb{R}^{d_{i+1} \times d_{i-1}}$. Using Observation 1, we conclude that any local minimum of problem (33) relates to a local minimum of the problem obtained by replacing $\bbW_{i+1}\bbW_{i}$ by $\bar{\bZ}_{i+1,i} \in \mathbb{R}^{d_{i+1} \times d_{i-1}}$. By a similar argument, we conclude that if $\bbW_i$ is a full row rank for some $2 \leq i \leq h$, any local minimum of problem (33) relates to local minimum of the problem obtained by replacing $\bbW_{i}\bbW_{i-1}$ by $\bar{\bZ}_{i,i-1} \in \mathbb{R}^{d_{i} \times d_{i-2}}$. Thus, if $\bbW=(\bbW_h, \ldots, \bbW_1)$ is a local minimum of problem (33), the new point 
	$\bar{\bZ}=(\bar{\bZ}_{h^{\prime}}^{\prime}, \ldots , \bar{\bZ}_1^{\prime})$, where $\bar{\bZ}_i \in \mathbb{R}^{d_i^{\prime} \times d_{i-1}^{\prime}} $ and $h^{\prime} \leq h$, is a local minimum of the problem attained by applying the replacements discussed above. If
	$h^{\prime} = 1$, we get the desired result from Lemma 7. Else, if $h^{\prime} = 2$, the auxiliary problem becomes a two layer linear network for which Theorem 8 provides the desired result. When $h^{\prime} >2$, examine $d_{h^{\prime}}^{\prime}, d_{h^{\prime}-1}^{\prime}, d_1^{\prime}$ and $d_0^{\prime}$. If $d_{h^{\prime}}^{\prime} > d_{h^{\prime}-1}^{\prime}$ and $d_{0}^{\prime}> d_{1}^{\prime}$, then there exist $1 \leq p_1 < p_2 \leq h-1$ with $d_h > d_{p_2}$ and $d_0 > d_{p_1}$
	which contradicts our assumption. It follows by construction of $\bar{\bZ}_i$, that either $d_{h^{\prime}}^{\prime} \leq  d_{h^{\prime}-1}^{\prime}$ 
	and $\bar{\bZ}_{h^{\prime}}^{\prime}$ is not full row rank or $d_{0}^{\prime} \leq d_{1}^{\prime}$ and $\bar{\bZ}_{1}^{\prime}$ is not full column rank; thus at least one of the null spaces ${\cal N}\big( \bZ_{h^{\prime}}^{\prime}\, \big)$, ${\cal N}\big((\bar{\bZ}_{1}^{\prime})^{\top} \big)$ is non empty. Moreover, $\bar{\bZ}_i$ has non-empty right and left null spaces for $2 \leq i \leq h-1$. The result follows using Lemma~\ref{lm:degenerate-local-global}.
 \end{proof}

\end{appendices}

\end{document}